\documentclass[11pt]{article}
\usepackage{graphicx}
\usepackage{amsmath,amssymb,amsthm,amsfonts}
\usepackage{amssymb}
\usepackage{cite}
\usepackage{subfig}
\usepackage{amsmath}

\usepackage[hidelinks]{hyperref}

\allowdisplaybreaks[4]
\newtheorem{thm}{Theorem}[section]

\newtheorem{lem}[thm]{Lemma}
\newtheorem{prop}[thm]{Proposition}

\theoremstyle{definition}

\theoremstyle{remark}
\newtheorem{rem}{Remark}[section]

\numberwithin{equation}{section}

\DeclareMathSymbol{\C}{\mathalpha}{AMSb}{"43}

\makeatletter
\renewenvironment{proof}[1][\proofname]{\par
  \pushQED{\qed}%
  \normalfont\topsep6\p@\@plus6\p@\relax
  \trivlist
  \item[\hskip\labelsep\normalfont\bfseries #1\@addpunct{.}]\ignorespaces
}{%
  \popQED\endtrivlist\@endpefalse
}
\makeatother

\newcommand{\R}{\mathbb{R}}
\newcommand{\norm}[1]{\left\lVert #1\right\rVert}
\newcommand{\abs}[1]{\left\lvert #1\right\rvert}
\newcommand{\Hrad}{H^1_{O(N)}(\R^N)}
\newcommand{\Hmrad}{H^{-1}_{O(N)}(\R^N)}

\begin{document}

\title{Uniqueness of Bound States for Double-Power Scalar Field Equations}

\author{ Yujin Guo\thanks{School of Mathematics and Statistics, and  Key Laboratory of Nonlinear Analysis $\&$ Applications (Ministry of Education), Central China Normal University, P.O. Box 71010, Wuhan 430079, P. R. China. Email: \texttt{yguo@ccnu.edu.cn}. Y. J. Guo
is partially supported by NSF of China (Grants 12225106 and 12371113) and National Key R $\&$ D Program of China (Grant 2023YFA1010001).}
,\, 	Wei Ji\thanks{School of Mathematics and Statistics, and  Key Laboratory of Nonlinear Analysis $\&$ Applications 
(Ministry of Education), Central China Normal University, P.O. Box 71010, Wuhan 430079, P. R. China. 
Email: \texttt{jiwei2020@amss.ac.cn}. }
,\, and\, 	Yong Luo\thanks{School of Mathematics and Statistics,  and  Hubei Key Laboratory of
Mathematical Sciences, Central China Normal University, P.O. Box 71010, Wuhan 430079, P. R. China. Email: \texttt{yluo@ccnu.edu.cn}.
Y. Luo is partially supported by NSF of China (Grant 12571119).}
}
\date{}
\smallbreak\maketitle

\begin{abstract}
We study the uniqueness of  bound states for the following scalar field equation
\[
 -\Delta u+u=-\varepsilon |u|^{q-1}u+|u|^{p-1}u
\,\ \ \mbox{in}\,\ \mathbb R^N,
 \,\ \varepsilon>0,\eqno{(*)}
\]
where \(  N\ge3,  1<p<(N+2)/(N-2)
\) and \(  q>1 \) satisfies $q \neq p$. For each integer $k \geq 1$, we prove that
the equation $(*)$ has, up to translation and reflection, a unique bound state with precisely $k$ nodes in a natural coupling
regime, where $\varepsilon>0$ is sufficiently small. Our results address a conjecture of
[H. Berestycki and P. L. Lions, ARMA, 1983][M. X. Tang, Invent. Math., 2026] on the uniqueness of fixed-node bound states for the equation $(*)$.
\end{abstract}

\vskip 0.05truein

\noindent {\it Keywords:} Double-power scalar field equation; Bound states; Uniqueness

\vskip 0.2truein

\section{Introduction}
\label{sec-main}
In this paper, we study bound states
of the following double-power scalar field equation
\begin{equation}\label{mainequation}
 -\Delta u+u=-\varepsilon |u|^{q-1}u+|u|^{p-1}u
 \,\ \ \text{in}\,\ \mathbb R^N,  \ \
 u(x)\in H^1(\mathbb R^N),
 \,\ \varepsilon>0,
\end{equation}
where
\begin{equation}\label{eq-range}
 N\ge3,
 \ \
 1<p<\frac{N+2}{N-2},
 \ \
 q>1,
 \ \
 q\ne p.
\end{equation}
Equation \eqref{mainequation} arises from standing waves of the nonlinear
Schr\"odinger equation
\begin{equation}\label{eq-time-dependent-NLS}
 i\partial_t\psi+\Delta\psi+|\psi|^{p-1}\psi
 -\varepsilon|\psi|^{q-1}\psi=0\,\ \ \text{in}\,\ (0, +\infty)\times\mathbb R^N,
\end{equation}
where \(\psi(t,x)=e^{it}u(x)\), and two nonlinear terms
represent competitive focusing and defocusing effects.
Such equations are used to model the nucleations, dispersive media and dilute quantum gases, see
\cite{Coleman1977,DesyatnikovMaimistovMalomed2000,gross1961structure,MorozMuratov2014,pitaevskii1961vortex}
and the references therein.
In particular, the cubic--quintic case \(p=3\) and \(q=5\) of \eqref{mainequation} appears in the
systems of attractive two-body and repulsive three-body interactions
\cite{Anderson1971,Kartavenko1984,MerchantIsidroFilho1988},
spatiotemporal optical solitons
\cite{DesyatnikovMaimistovMalomed2000}, and trapped
Bose--Einstein condensates
\cite{GammalFredericoTomioChomaz2000}.
Moreover, the case $p=5/3$, $q=7/3$ and $N=3$ of \eqref{mainequation} 
was considered (cf. \cite{Ricaud2018}) in the periodic Thomas--Fermi--Dirac--von Weizs\"acker theory.

The existence theory of solutions for \eqref{mainequation} is related to that of the
following more general scalar field equation
\begin{equation}\label{eq:general-scalar-field}
 \Delta u+f(u)=0
\,\ \ \text{in }\ \mathbb R^N,
\ \ \text{where} \ \
 u(x)\in H^1(\mathbb R^N).
\end{equation}
When $N\ge3$, Berestycki and Lions
\cite{berestycki1983nonlinearI,berestycki1983nonlinearII} considered
nontrivial solutions of \eqref{eq:general-scalar-field} under the assumptions that
$f:\mathbb R\to\mathbb R$ is continuous, odd and satisfies
\begin{enumerate}
 \item[(H1)]
 \[
  -\infty<\liminf_{s\to0}\frac{f(s)}{s}
  \leq\limsup_{s\to0}\frac{f(s)}{s}=-m<0,
 \]
 \item[(H2)]
 \[
  -\infty\leq
  \limsup_{s\to+\infty}\frac{f(s)}{s^{\frac{N+2}{N-2}}}
  \leq0,
 \]
 \item[(H3)] there exists $\xi>0$ such that
 \[
  F(\xi):=\int_0^\xi f(s)\,ds>0.
 \]
\end{enumerate}
Under the assumptions (H1)--(H3), the authors proved in \cite{berestycki1983nonlinearI,berestycki1983nonlinearII} that
\eqref{eq:general-scalar-field} admits infinitely many distinct
radial solutions, at least one of which is positive, and all of
which decay exponentially as $|x|\to+\infty$. We also refer
 to \cite{Strauss1977,BerestyckiLionsPeletier1981} for further results on the existence of nontrivial
radial solutions for \eqref{eq:general-scalar-field}.
One can note that  the nonlinearity
\begin{equation}
 \label{1:f}
 f(s)=-s-\varepsilon\abs{s}^{q-1}s+\abs{s}^{p-1}s
 \end{equation}
satisfies (H1) and (H2) under \eqref{eq-range}
for any $\varepsilon>0$, while it satisfies (H3) if and only if
$0<\varepsilon<\varepsilon_*$, where
\begin{equation}
 \label{eq-existence-threshold}
 \varepsilon_*:=
 \begin{cases}
  +\infty, & 1<q<p;\\
  \displaystyle
  \frac{(p-1)(q+1)}{(p+1)(q-1)}
  \left[
   \frac{2(q-p)}{(p+1)(q-1)}
  \right]^{\frac{q-p}{p-1}}>0,
  & q>p>1.
 \end{cases}
\end{equation}
Applying \cite[Theorem~1 and Example~2]{berestycki1983nonlinearI}, it thus follows from
\cite{berestycki1983nonlinearII} that
\eqref{mainequation} admits infinitely many distinct
nontrivial radial solutions
for any $\varepsilon\in(0,\varepsilon_*)$,
at least one of which is positive in $\mathbb R^N$.

The uniqueness of positive solutions for \eqref{eq:general-scalar-field} was studied in
\cite{coffman1972uniqueness,mcleod1987uniqueness,kwong1989uniqueness} for the pure-power nonlinearity
$f(u)=-u+\abs{u}^{p-1}u$, where $N\ge3$ and $1<p<\frac{N+2}{N-2}$.
The uniqueness of positive
solutions for \eqref{eq:general-scalar-field} under other assumptions on $f(u)$ was further analyzed in
\cite{cortazar1998uniqueness,kabeya1999uniqueness,
kwong1992uniqueness,McLeod1993, serrin2000uniqueness,shioji2013generalized} and the references therein.
Specially for $N\ge3$, Serrin and Tang
\cite{serrin2000uniqueness} proved the uniqueness,
up to translation, of positive
$H^1(\mathbb R^N)$ solutions for the following
double-power scalar field equation
\begin{equation}
 \label{eq-positive-double-power}
 -\Delta v+\omega v=v^p-v^q \,
 \,\ \ \text{in}\ \, \mathbb R^N,\ \ v(x)\in H^1(\mathbb R^N),
\end{equation}
where $1<p<q$ and
$0<\omega<\omega_*:=
\varepsilon_*^{\frac{p-1}{q-p}}$.
We should comment that equation \eqref{eq-positive-double-power}
can be equivalently reduced into the form \eqref{mainequation} by scaling,
where $\varepsilon=\omega^{\frac{q-p}{p-1}}>0$. When $N\ge2$, Lewin and Rota  Nodari
\cite{Lewin2020CV} further established the uniqueness
and  nondegeneracy of positive $H^1(\mathbb R^N)$ solutions for
\eqref{eq-positive-double-power}. The existing analytical results also show  that  
these positive solutions of \eqref{eq-positive-double-power} depend on $\omega$. Specially for $N\ge3$, Moroz and Muratov
\cite{MorozMuratov2014} studied the asymptotic behavior of positive solutions for
\eqref{eq-positive-double-power} as $\omega\searrow0$.
When $N\ge2$, Lewin and Rota Nodari \cite{Lewin2020CV}
further analyzed the $L^2$-mass and asymptotic
behavior  of positive solutions for \eqref{eq-positive-double-power}
as $\omega\searrow0$. We remark that a positive solution of \eqref{eq:general-scalar-field}  is often called
a ground state of \eqref{eq:general-scalar-field}.

It follows from \cite{CohenLiSchlag2024,Rosen1952,tang2026uniqueness} that
{\em a sign-changing radial solution} of \eqref{eq:general-scalar-field} in
$H^1(\mathbb R^N)\cap L^\infty(\mathbb R^N)$ is often called
{\em a bound  state (or  excited state)}   of \eqref{eq:general-scalar-field}.
Even though Berestycki--Lions theorem in
\cite{berestycki1983nonlinearI,berestycki1983nonlinearII}
gives infinitely many nontrivial radial solutions
of \eqref{eq:general-scalar-field},
these works  do not  imply in  a direct way that
for each integer $k\ge1$,
\eqref{eq:general-scalar-field} admits a bound state
with precisely $k$ nodes.
For this reason, the authors in
\cite{JonesKupper1986,McLeodTroyWeissler1990,Grillakis1990}
proved the existence of bound states with precisely
$k$ nodes for the following pure-power nonlinear equation
\begin{equation}
 \label{eq-limit}
 -\Delta Q+Q=\abs{Q}^{p-1}Q
 \ \,\ \text{in}\ \, \mathbb R^N, \ \ Q(x)\in H^1(\mathbb R^N),
\end{equation}
where $N\ge3$ and $1<p<\frac{N+2}{N-2}$.
As for the double-power equation \eqref{mainequation} in the range \eqref{eq-range}, it was recently proved in
\cite[Theorem~1.1]{CortazarGarciaHuidobroYarur2013} that  for each integer \(k\geq1\) and any
\(\varepsilon\in(0,\varepsilon_*)\), \eqref{mainequation} admits a bound state with precisely
\(k\) nodes, where $\varepsilon_*>0$ is defined by
\eqref{eq-existence-threshold}.

The uniqueness and multiplicity of bound states with precisely \(k\) nodes for \eqref{eq:general-scalar-field} 
have attracted a lot of attention since the past few decades. More precisely, under suitable assumptions on $f(u)$,
Cort\'azar, Garc\'ia-Huidobro and Herreros  proved in \cite{Herreros2015} that
for each integer $k\ge1$, \eqref{eq:general-scalar-field} admits at least two
distinct bound states  with precisely $k$ nodes.
On the other hand, under different assumptions on $f(u)$,
Cort\'azar, Garc\'ia-Huidobro and Yarur analyzed in
\cite{CortazarGarciaHuidobroYarur2011} that for any $2\le N\le4$ and each integer $k\ge1$,
\eqref{eq:general-scalar-field} admits at most one bound state with precisely $k$ nodes.
Moreover, for the pure-power nonlinearity
$f(u)=-u+\abs{u}^{p-1}u$, Ao, Wei and Yao
proved in \cite{Wei2016} that for any $N\ge3$
and each fixed integer $k\ge1$,
\eqref{eq:general-scalar-field} admits a unique
bound state with precisely $k$ nodes,
if $p<\frac{N+2}{N-2}$ is sufficiently close to
$\frac{N+2}{N-2}$.
  Cohen, Li and Schlag
recently employed in \cite{CohenLiSchlag2024} a rigorous computer-assisted argument
to prove that for each integer $k$ satisfying
$1\le k\le20$, the equation
\begin{equation}
 \label{eq:cubic-excited-states}
 -\Delta u+u=u^3
 \,\ \ \text{in}\ \ \mathbb R^3
\end{equation}
admits a unique bound state with precisely $k$ nodes.


More generally, Berestycki and Lions
\cite{berestycki1983nonlinearI,berestycki1983nonlinearII}
conjectured in 1983 that for certain classes of
nonlinearities $f(u)$,
\eqref{eq:general-scalar-field} admits a unique (up to translation and reflection)
bound state with precisely $k$ nodes, where the integer $k\ge1$ is arbitrary.
Tang recently proved in \cite{tang2026uniqueness} that
for any $N\ge3$ and $1<p<\frac{N+2}{N-2}$,
\eqref{eq-limit} admits a unique (up to translation and reflection)
bound state with  precisely $k$ nodes, where the integer $k\ge1$ is arbitrary.
This resolves the above mentioned Berestycki--Lions conjecture (cf. \cite{berestycki1983nonlinearI,berestycki1983nonlinearII})
for the pure-power equation \eqref{eq-limit}.
At the same time, Tang also conjectured in \cite{tang2026uniqueness} that
for $N\ge2$ and each integer $k\ge1$,
\eqref{eq:general-scalar-field} admits at most one (up to translation and reflection)
bound state with  precisely $k$ nodes,
when the nonlinearity $f(\cdot)$ is given by
\begin{equation}
 \label{eq:tang-double-power}
 f(u)=-u-\lambda\abs{u}^{q-1}u+\mu\abs{u}^{p-1}u,
 \ \ \lambda,\mu>0,
\end{equation}
where $1<p,q<2^*_N:=\frac{N+2}{N-2}$ satisfy $p\ne q$, and we define $2^*_2:=+\infty$.
One can observe that if $u(x)$ is a solution of
\eqref{eq:general-scalar-field} satisfying \eqref{eq:tang-double-power},
then $\mu^{\frac1{p-1}}u(x)$ is a solution of
\eqref{mainequation}, where
\[
 \varepsilon :=\lambda\mu^{-\frac{q-1}{p-1}}>0.
\]
Therefore, the analysis of bound states for \eqref{mainequation} gives equivalently the analytical results of bound states for \eqref{eq:general-scalar-field} satisfying \eqref{eq:tang-double-power}.

Stimulated by the above conjectures of \cite{berestycki1983nonlinearI,berestycki1983nonlinearII,tang2026uniqueness}, the main purpose of the present paper is to prove the   uniqueness (up to translation and reflection) of $k$-node bound states   for \eqref{mainequation} satisfying (\ref{eq-range}),  where $\varepsilon>0$ is sufficiently small. More precisely, by analyzing the limiting behavior 
of $k$-node bound states for \eqref{mainequation} satisfying (\ref{eq-range}) as $\varepsilon \searrow 0$, the main result of the present paper is to establish the following uniqueness theorem.

\begin{thm} \label{maintheorem}
Suppose $p, q>1$ satisfy the assumption (\ref{eq-range}). Then for each integer $k \geq 1$, there 
exists  $\varepsilon_k=\varepsilon (k)>0$ such that  for any $0<\varepsilon<\varepsilon_k$,  \eqref{mainequation} 
has, up to translation and reflection, a unique bound state with precisely $k$ nodes.
\end{thm}

\begin{rem}
The equation investigated in Theorem \ref{maintheorem} is essentially equivalent to the form (\ref{eq:general-scalar-field}) 
satisfying (H1)--(H3) in a natural coupling regime (\ref{eq-existence-threshold}). The uniqueness of  Theorem \ref{maintheorem} 
therefore addresses the conjectures of Berestycki $\&$ Lions in \cite{berestycki1983nonlinearI,berestycki1983nonlinearII} and Tang in \cite{tang2026uniqueness}, which are concerned with the uniqueness of fixed-node bound states for \eqref{mainequation} in a natural coupling regime.
Of course, the conjecture of \cite{tang2026uniqueness}  requires additionally that $1< q<2^*_N:=\frac{N+2}{N-2}$. 
Therefore, we note from (\ref{eq-range}) and  \eqref{eq:tang-double-power} that  
the uniqueness of Theorem \ref{maintheorem}  holds for a wider range of $q>1$, 
compared with Tang's conjecture in \cite{tang2026uniqueness}.
\end{rem}

We point out that the existence of Theorem \ref{maintheorem} is due to the work \cite{CortazarGarciaHuidobroYarur2013}, 
while the uniqueness argument of \cite{tang2026uniqueness} is not applicable to Theorem  \ref{maintheorem}. 
In fact, the uniqueness argument of \cite{tang2026uniqueness} is applied to the equation \eqref{eq:general-scalar-field}, 
where the nonlinear term $f(u)$ is requested to satisfy  $uf'(u)-f(u)>0$ for any $u>0$, see \cite[Lemma~4.2]{tang2026uniqueness}.
This character is unfortunately false for the equation \eqref{mainequation} under the assumption (\ref{eq-range}). 
In order to prove Theorem \ref{maintheorem}, roughly speaking, we fix each integer $k\ge1$ and let $\varepsilon _n>0$ 
be sufficiently small, where $\varepsilon _n \searrow 0$ as $n\to\infty$.
We shall compare $u_{n,k}(x)$ with $Q_k(x)$
in $L^\infty(\R^N)$ as $n \to \infty$,
where $u_{n,k}(x)$ and $Q_k(x)$ are
$k$-node bound states centered at the origin  of \eqref{mainequation}
with $\varepsilon=\varepsilon_n$
and of \eqref{eq-limit}, respectively.
Towards this aim, we shall employ Tang's uniqueness  (cf. \cite{tang2026uniqueness}) of $Q_k(x)$, 
and we also establish in Lemma 4.1 the nondegeneracy of  $Q_k(x)$ in $H^1_{O(N)}(\R^N)$, 
where the radial space $H^1_{O(N)}(\R^N)$ is defined by
\begin{equation}
 \label{1:radial}
 \begin{aligned}
  \Hrad
:=\{u\in H^1(\R^N):\, 
u\circ T=u
\ \ \text{for any}\ \ T\in O(N)\}, 
 \end{aligned}
\end{equation}
and $O(N)$ denotes the orthogonal group on $\R^N$.

For any $u\in\Hrad$, we also denote $u(x)$ by $u(r)$,
where $r=|x|$.
To sketch the proof of Theorem \ref{maintheorem}, we now denote $\mathcal B_{\varepsilon,k}$ the set
of $k$-node bound states $u(r)$ for \eqref{mainequation}
centered at the origin and satisfying $u(0)>0$,
and let $Q_k(r)$ be the unique $k$-node bound state
of \eqref{eq-limit} centered at the origin
and satisfying $Q_k(0)>0$.
We also denote $\mathcal B_{\varepsilon,0}$
the set of positive solutions for \eqref{mainequation}
in $\Hrad\cap L^\infty(\R^N)$,
and let $Q_0(r)$ be the unique positive solution
of \eqref{eq-limit} in $\Hrad\cap L^\infty(\R^N)$.
It then yields from
\cite[Theorem~1.1]{CortazarGarciaHuidobroYarur2013}
that
\[
 \mathcal B_{\varepsilon,k}\ne\varnothing
 \ \ \text{for each integer}\ \ k\ge1
 \ \ \text{and any}\ \ 0<\varepsilon<\varepsilon_*.
\]
We shall carry out the proof of
Theorem~\ref{maintheorem} by the following three steps.

As the first step of proving Theorem~\ref{maintheorem}, we shall mainly establish the following uniform convergence of
all bound states $u_\varepsilon \in\mathcal B_{\varepsilon,k}$
for \eqref{mainequation} as $\varepsilon\searrow0$.

\begin{thm} \label{thm-main}
Suppose $u_\varepsilon\in\mathcal B_{\varepsilon,k}$ is a $k$-node bound state of \eqref{mainequation}, where
 $p, q>1$ satisfy the assumption (\ref{eq-range}). Then for each integer $k \geq 1$, there exist
 $\varepsilon_k\in(0,\varepsilon_*)$ and $C_k>0$
 such that for any $0<\varepsilon<\varepsilon_k$,
 \begin{equation}
\label{eq-uniform-bounds}
\sup_{u_\varepsilon\in\mathcal B_{\varepsilon,k}}
  \left(
  \norm{u_\varepsilon}_{H^1(\R^N)}
  +\norm{u_\varepsilon}_{L^\infty(\R^N)}
  \right)
  \le C_k.
 \end{equation}
 Moreover,
 \begin{equation}
\label{eq-uniform-convergence}
\sup_{u_\varepsilon\in\mathcal B_{\varepsilon,k}}
  \left(
  \norm{u_\varepsilon-Q_k}_{H^1(\R^N)}
  +\norm{u_\varepsilon-Q_k}_{L^\infty(\R^N)}
  \right)
  \to0
  \ \ \text{as}\ \ \varepsilon\searrow0.
 \end{equation}
\end{thm}

\noindent
The key of proving Theorem \ref{thm-main} is to derive
a uniform $L^\infty(\R^N)$ bound
of all bound states $u_\varepsilon(r)\in\mathcal B_{\varepsilon,k}$
for \eqref{mainequation},
which then yields that up to a subsequence if necessary,
the sequence $u_n(r)=u_{\varepsilon_n}(r) \in\mathcal B_{\varepsilon_n,k}$ satisfies
\[
u_n(r)\to Q_m(r)\ \ \text{in}\ \ C^2_{\mathrm{loc}}\bigl([0,+\infty)\bigr)\ \ \text{as}\ \ \varepsilon_n\searrow0,
\]
where the integer $m$ satisfies $0\le m\le k$.
If $0\le m<k$, then we shall prove that some zero point  of $u_n(r)$
 tends to $+\infty$ as $n\to\infty$. By deriving a radial Pohozaev identity of \eqref{mainequation},
we are however able to prove that all zero points of $u_n(r)$
remain bounded uniformly in $n>0$, which is a contradiction. This hence proves that $m=k$ and
$Q_m(r)\equiv Q_k(r)$ in $[0,+\infty)$.

The second step of proving Theorem~\ref{maintheorem}
is to derive the following inverse estimate
\begin{equation}
 \label{1:eq-Lk}
 \norm{\phi}_{H^1(\R^N)}
 \le A_k\norm{L_k\phi}_{H^{-1}(\R^N)},
 \ \ \forall \, \phi\in\Hrad,
\end{equation}
where the constant $A_k>0$ is independent of $\phi$, and the radial
linearized operator $L_k: \Hrad \to\Hmrad$ is defined by
\begin{equation}
 \label{eq-Lk}
 L_k:=-\Delta+1-p\abs{Q_k(x)}^{p-1}
 \,\ \ \text{in}\ \, \R^N.
\end{equation}
Here $\Hrad$ is defined in \eqref{1:radial},
and $\Hmrad:=(\Hrad)^*$.
More precisely, applying  the shooting variation of
\cite[Lemma~3.1(b) and Section~5.8]{tang2026uniqueness},
together with the ODE uniqueness theorem,
we shall first prove that any nonzero radial kernel
$\phi\in C^2(\R^N)$ of the operator $L_k$ satisfies
$\abs{\phi(x)}\to+\infty$ as $\abs{x}\to+\infty$.
Thus, the operator $L_k$ does not have nonzero kernel
$\phi\in\Hrad$, which  proves the nondegeneracy
of $Q_k$ in $\Hrad$.
Applying this nondegeneracy, in this step we finally derive the inverse estimate (\ref{1:eq-Lk}).

As the last step of proving Theorem~\ref{maintheorem},
on the contrary, suppose that there exists a sequence
$\{\varepsilon_n\}\subset(0,\varepsilon_*)$
satisfying $\varepsilon_n\searrow0$ as $n\to\infty$
such that for each $n\ge1$, \eqref{mainequation}
with $\varepsilon=\varepsilon_n$ admits two different
$k$-node bound states
$u_n,w_n\in\mathcal B_{\varepsilon_n,k}$.
Fix $\eta_k>0$ to be sufficiently small.
It then follows from Theorem~\ref{thm-main} that for any sufficiently large $n>0$, 
\begin{equation}
\label{1A:eq-Lk}
\norm{u_n-Q_k}_{L^\infty(\R^N)}<\eta_k,
 \ \ \norm{w_n-Q_k}_{L^\infty(\R^N)}<\eta_k. 
\end{equation}
In Section~\ref{sec-uniqueness},
we derive from \eqref{1A:eq-Lk} that for any sufficiently large $n>0$,
\[
 \norm{L_k(u_n-w_n)}_{H^{-1}(\R^N)}
 \le\frac1{2A_k}\norm{u_n-w_n}_{H^1(\R^N)},
\]
where the constant $A_k>0$ is as in \eqref{1:eq-Lk}.
Since $u_n-w_n\in\Hrad$,
we deduce from \eqref{1:eq-Lk} that for any sufficiently large $n>0$,
\[
 \begin{aligned}
  0<\norm{u_n-w_n}_{H^1(\R^N)}
  &\le A_k\norm{L_k(u_n-w_n)}_{H^{-1}(\R^N)}\\
  &\le\frac12\norm{u_n-w_n}_{H^1(\R^N)}, 
 \end{aligned}
\]
a contradiction. 
We therefore conclude that there exists
$\varepsilon_k\in(0,\varepsilon_*)$
such that $\mathcal B_{\varepsilon,k}$
contains exactly one element
for any $0<\varepsilon<\varepsilon_k$.
This completes in Section~\ref{sec-uniqueness} the proof of
Theorem~\ref{maintheorem}.

In Section \ref{sec-uniqueness}, we also obtain in Proposition~\ref{cor-perturbed-nondegeneracy}
that for any sufficiently small $\varepsilon>0$, the unique $k$-node bound state
$u_\varepsilon\in\mathcal B_{\varepsilon,k}$
of \eqref{mainequation} is nondegenerate in $\Hrad$, 
in the sense that
\[
 \ker_{\Hrad}\left(
 -\Delta+1-p\abs{u_\varepsilon(x)}^{p-1}
 +\varepsilon q\abs{u_\varepsilon(x)}^{q-1}
 \right)=\{0\}, 
\]
where $\Hrad$ is defined by \eqref{1:radial}.


This paper is organized as follows. In Section~\ref{sec-upper-bound}, we establish
the uniform $L^\infty(\R^N)$ estimates of $k$-node bound states for \eqref{mainequation} satisfying \eqref{eq-range}
as $\varepsilon\searrow0$. Section~\ref{sec-noescape} is devoted to the proof of Theorem~\ref{thm-main} 
on the uniform convergence of $k$-node bound states for \eqref{mainequation} satisfying \eqref{eq-range} as $\varepsilon\searrow0$.
The proof of Theorem~\ref{maintheorem} is finally completed in Section~\ref{sec-uniqueness}.

\section{Uniform Estimates of Bound States as
\texorpdfstring{$\varepsilon\searrow0$}{epsilon decreases to 0}}
\label{sec-upper-bound}
For each integer $k\ge1$,
from this section on, we always suppose 
$u_\varepsilon(r)$ to be a $k$-node bound state
of \eqref{mainequation} satisfying \eqref{eq-range},
where $\varepsilon\in(0,\varepsilon_*)$
and $\varepsilon_*$ is defined by
\eqref{eq-existence-threshold}.
Recall that $\mathcal B_{\varepsilon,k}$ denotes
the set of $k$-node bound states $u_\varepsilon(r)$
for \eqref{mainequation} centered at the origin and satisfying $u_\varepsilon(0)>0$.
The main purpose of this section is to establish the following proposition, 
which is concerned with a uniform $L^\infty$ estimate of $u_\varepsilon(r)$ as $\varepsilon\searrow0$.

\begin{prop}
 \label{prop-amplitude}
Under the assumption (\ref{eq-range}), suppose $u_\varepsilon\in\mathcal B_{\varepsilon,k}$ is a $k$-node bound state 
of \eqref{mainequation}.  Then  there exist $M_k>0$  and $\varepsilon_k\in(0,\varepsilon_*)$ such that  for any $0<\varepsilon<\varepsilon_k$,
 \begin{equation}
\label{eq-amplitude-bound-qgeqp}
\norm{u_\varepsilon}_{L^\infty(\mathbb R^N)}\le M_k.
 \end{equation}
\end{prop}

  For any $\varepsilon\in [0,\varepsilon_*)$, where
$\varepsilon_*$ is defined by \eqref{eq-existence-threshold}, define
\begin{equation}\label{eq-fF}
\begin{aligned}
f_\varepsilon(s)
&:= -s-\varepsilon\abs{s}^{q-1}s+\abs{s}^{p-1}s,\\
F_\varepsilon(s)
&:= -\frac{s^2}{2}
-\frac{\varepsilon\abs{s}^{q+1}}{q+1}
+\frac{\abs{s}^{p+1}}{p+1},
\end{aligned}
\end{equation}
where $p,q>1$ satisfy \eqref{eq-range}, so that
\begin{equation}
 \label{eq-F0}
 F_0(s):=-\frac{s^2}{2}
 +\frac{\abs{s}^{p+1}}{p+1},
 \ \ s\in\mathbb R.
\end{equation}
Note that $F_0(\cdot)$ has a unique positive zero point
\begin{equation}\label{eq-definebeta}
\beta_0:=\left(\frac{p+1}{2}\right)^{\frac1{p-1}}>0.
\end{equation}
For fixed $\varepsilon\in(0,\varepsilon_*)$,
we denote $u_\varepsilon(r)$ by $u(r)$ for convenience.
Any bound state $u=u(r)$ of \eqref{mainequation}
satisfies the following ordinary differential equation
\begin{equation}
 \label{eq-radial}
 \begin{cases}
  u''+\frac{N-1}{r}u'+f_\varepsilon(u)=0   \,\ \ \text{in}\ \ r>0, \\
  u'(0)=0,
  \ \ u(r)\to0
  \ \ \text{as}\ \ r\to\infty, 
 \end{cases}
\end{equation}
where $f_\varepsilon(\cdot)$ is as in \eqref{eq-fF}. 
We also define the corresponding energy function of \eqref{eq-radial} by
\begin{equation}
 \label{eq-defineH}
 H_\varepsilon(r)
 :=\frac12\abs{u'(r)}^2+F_\varepsilon(u(r)),
 \ \ r\ge0, 
\end{equation}
where $F_\varepsilon(\cdot)$ is as in \eqref{eq-fF}. 
Since $u(r)\to0$ as $r\to\infty$, we deduce from
\eqref{eq-radial} that $u'(r)\to0$ as $r\to\infty$.
It then follows from \eqref{eq-radial} and
\eqref{eq-defineH} that
\begin{equation}
 \label{eq-Hamiltonian-derivative}
 H_\varepsilon'(r)
 =-\frac{N-1}{r}\abs{u'(r)}^2\le0
 \ \ \text{for any}\ \ r>0,
 \ \ \lim_{r\to\infty}H_\varepsilon(r)=0,
\end{equation}
which further imply that $H_\varepsilon(r)\ge0$ holds
for any $r\ge0$.

We claim that
\begin{equation}
 \label{eq-H-positive}
 H_\varepsilon(r)>0
 \ \ \text{for any}\ \ r\ge0.
\end{equation}
On the contrary, suppose that $H_\varepsilon(r_0)=0$ holds
for some $r_0\ge0$.
It then follows from \eqref{eq-Hamiltonian-derivative} that
\[
 H_\varepsilon(r)=0, \ \ \text{for any}\ \ r>r_0, 
 \]
 and hence $H'_\varepsilon(r)=0$ for any $r>r_0$. This shows that
 $u'(r) \equiv 0$ for any $r>r_0$.
The decaying property of \eqref{eq-radial} further implies that
$u(r)\equiv 0$ for any $r>r_0$. Applying the ODE uniqueness theorem, we conclude that $u(r)\equiv0$ in $(0, +\infty)$,
which is a contradiction. This thus proves the claim \eqref{eq-H-positive}.

In order to prove Proposition \ref{prop-amplitude}, 
we now derive the following properties of \(f_\varepsilon(s)\) and
\(F_\varepsilon(s)\).

\begin{lem}
 \label{lem-scalar-structure}
 For any $\varepsilon\in (0,\varepsilon_*)$,
let $f_\varepsilon(s)$ and $F_\varepsilon(s)$ be defined by \eqref{eq-fF}, where $p,q>1$ satisfy \eqref{eq-range},
  and $\varepsilon_*$
 is defined by \eqref{eq-existence-threshold}.
 Then there exist $a_\varepsilon,\beta_\varepsilon,
 b_\varepsilon$ satisfying
\begin{equation}
  \label{A:eq-root-limits}
  1<a_\varepsilon<\beta_\varepsilon
  <b_\varepsilon\leq+\infty,
 \end{equation}
and
 \begin{equation}
  \label{eq-root-limits}
  a_\varepsilon\to1,
  \ \ \beta_\varepsilon\to\beta_0=\left(\frac{p+1}{2}\right)^{\frac1{p-1}},
  \ \ b_\varepsilon\to+\infty
  \ \ \text{as}\ \ \varepsilon\searrow0,
 \end{equation}
 where $b_\varepsilon<+\infty$ if $q>p$, and
 $b_\varepsilon=+\infty$ if $q<p$, such that
 \begin{equation}
  \label{eq-f-sign-high}
  f_\varepsilon(s)
  \begin{cases}
   <0, & s\in(0,a_\varepsilon);\\
   >0, & s\in(a_\varepsilon,b_\varepsilon);\\
   <0, & s\in(b_\varepsilon,+\infty);
  \end{cases}
 \end{equation}
 \begin{equation}
  \label{eq-F-sign-high}
  F_\varepsilon(s)
  \begin{cases}
   <0, & s\in(0,\beta_\varepsilon);\\
   >0, & s\in(\beta_\varepsilon,b_\varepsilon);
  \end{cases}
 \end{equation}
 and $F_\varepsilon(s)$ is strictly increasing on the interval
 \begin{equation}
  \label{Ivarepsilon}
  \mathcal I_\varepsilon:=(a_\varepsilon,b_\varepsilon).
 \end{equation}
\end{lem}

\begin{proof}
Define for any $\varepsilon\in(0,\varepsilon_*)$,
\[
 h_\varepsilon(s):=s^{p-1}-\varepsilon s^{q-1},
 \ \ s\ge0.
\]
It then yields from \eqref{eq-fF} that
$f_\varepsilon(s)=s(h_\varepsilon(s)-1)$ for any $s>0$,
and hence

\begin{equation}
 \label{eq-roots-equation}
 f_\varepsilon(s)=0,\ \ s>0;
 \ \ \Longleftrightarrow\ \
 h_\varepsilon(s)=1,\ \ s>0.
\end{equation}
Since $h_\varepsilon(s)<1$ holds for any $0<s\le1$,
we deduce from \eqref{eq-roots-equation} that any
positive zero point of $f_\varepsilon(s)$ belongs to
$(1,+\infty)$. We next consider the following two cases:

\textit{Case 1: $q>p$.}
For this case, we calculate that
\[
 h_\varepsilon'(s)
 =s^{p-2}[(p-1)-\varepsilon(q-1)s^{q-p}],
 \ \ s>0.
\]
It then follows that $h_\varepsilon(s)$ is strictly increasing
on $(0,\sigma_\varepsilon)$ and strictly decreasing
on $(\sigma_\varepsilon,+\infty)$, where
\[
 \sigma_\varepsilon
 :=\left(\frac{p-1}{(q-1)\varepsilon}\right)^{\frac1{q-p}}>0.
\]
Define
\[
 s_*:=
 \left(\frac{(p+1)(q-1)}{2(q-p)}\right)^{\frac1{p-1}}
 >\beta_0=\left(\frac{p+1}{2}\right)^{\frac1{p-1}}>1.
\]
It yields from \eqref{eq-existence-threshold} that
\[
 \varepsilon_*
 =\frac{(p-1)(q+1)}{2(q-p)}s_*^{1-q}.
\]
We then obtain from \eqref{eq-fF} that
\begin{equation}
 \label{eq-positive-test-point}
 f_\varepsilon(s_*)=(\varepsilon_*-\varepsilon)s_*^q>0,
 \ \ F_\varepsilon(s_*)
 =\frac{\varepsilon_*-\varepsilon}{q+1}s_*^{q+1}>0.
\end{equation}
Since $h_\varepsilon(0)=0$ and
\[
 h_\varepsilon(s)\to-\infty
 \ \ \text{as}\ \ s\to+\infty,
\]
applying the above monotonicity
of $h_\varepsilon(s)$, we deduce from \eqref{eq-roots-equation} and
\eqref{eq-positive-test-point} that $f_\varepsilon(s)$
has exactly two positive zero points $a_\varepsilon$ and
$b_\varepsilon$ satisfying
\[
 1<a_\varepsilon
 <b_\varepsilon<+\infty.
\]
This proves \eqref{eq-f-sign-high}. We further derive from
\eqref{eq-f-sign-high} and \eqref{eq-positive-test-point}  that $a_\varepsilon<s_*<b_\varepsilon$.
Since $F_\varepsilon'(s)=f_\varepsilon(s)>0$ holds for any $s\in(a_\varepsilon,b_\varepsilon)$,
we have
\[
 F_\varepsilon(b_\varepsilon)
 >F_\varepsilon(s_*)>0.
\]

\textit{Case 2: $q<p$.}
For this case, we calculate that
\[
 h_\varepsilon'(s)
 =s^{q-2}[(p-1)s^{p-q}-\varepsilon(q-1)],
 \ \ s>0.
\]
This yields that $h_\varepsilon(s)$ is strictly decreasing
on $(0,\widehat s_\varepsilon)$ and strictly increasing
on $(\widehat s_\varepsilon,+\infty)$, where
\[
 \widehat s_\varepsilon
 :=\left(\frac{\varepsilon(q-1)}{p-1}\right)^{\frac1{p-q}}>0.
\]
Since $h_\varepsilon(0)=0$ and
\[
 h_\varepsilon(s)\to+\infty
 \ \ \text{as}\ \, s\to+\infty,
\]
we deduce from \eqref{eq-roots-equation} that
$f_\varepsilon(s)$ has exactly one  positive zero point
$a_\varepsilon>1$.
Taking $b_\varepsilon=+\infty$, we then obtain 
\eqref{eq-f-sign-high}.
Moreover, it follows from \eqref{eq-fF} that
\[
 F_\varepsilon(s)\to+\infty
 \ \ \text{as}\ \, s\to+\infty.
\]

For both cases, we derive from \eqref{eq-fF} and
\eqref{eq-f-sign-high} that $F_\varepsilon(s)$ is
strictly decreasing on $(0,a_\varepsilon)$
and strictly increasing on $\mathcal I_\varepsilon$.
Moreover,
\[
 F_\varepsilon(a_\varepsilon)
 =\int_0^{a_\varepsilon}f_\varepsilon(s)\,ds<0,
 \ \ \lim_{s\nearrow b_\varepsilon}F_\varepsilon(s)>0.
\]
Since $F_\varepsilon(s)$ is continuous and strictly
increasing on $\mathcal I_\varepsilon$,
there exists a unique
$\beta_\varepsilon\in(a_\varepsilon,b_\varepsilon)$
satisfying $F_\varepsilon(\beta_\varepsilon)=0$.
It then follows that \eqref{eq-F-sign-high} holds.

We finally prove \eqref{eq-root-limits}.
For any fixed $t>1$, we have
\[
 h_\varepsilon(t)\to t^{p-1}>1
 \ \ \text{as}\ \ \varepsilon\searrow0.
\]
By \eqref{eq-fF} and \eqref{eq-f-sign-high}, this implies that
\[
 1<a_\varepsilon<t<b_\varepsilon,
\]
where $\varepsilon>0$ is sufficiently small.
Since $t>1$ is arbitrary, we obtain that
\begin{equation}
 \label{eq-endpoint-limits}
 a_\varepsilon\to1,
 \ \ b_\varepsilon\to+\infty
 \ \ \text{as}\ \ \varepsilon\searrow0.
\end{equation}
For any fixed $\delta\in(0,\beta_0-1)$, we deduce from
\eqref{eq-endpoint-limits} that
for any sufficiently small $\varepsilon>0$, 
\[
 a_\varepsilon<\beta_0-\delta
 <\beta_0+\delta<b_\varepsilon. 
\]
Note from \eqref{eq-fF} and \eqref{eq-definebeta} that
\[
  F_\varepsilon(\beta_0-\delta)
 \to F_0(\beta_0-\delta)<0,\ \
  F_\varepsilon(\beta_0+\delta)
 \to F_0(\beta_0+\delta)>0
 \ \ \text{as}\ \ \varepsilon\searrow0.
\]
Since $F_\varepsilon(s)$ is strictly increasing
on $\mathcal I_\varepsilon$, it yields that
\[
 \beta_0-\delta<\beta_\varepsilon<\beta_0+\delta,
\]
where $\varepsilon>0$ is sufficiently small.
Since $\delta>0$ can be arbitrarily small,
we conclude that $\beta_\varepsilon\to\beta_0$
as $\varepsilon\searrow0$, which hence proves \eqref{eq-root-limits}.
This completes the proof of Lemma~\ref{lem-scalar-structure}.
\end{proof}

Applying Lemma~\ref{lem-scalar-structure},
we next employ the monotonicity of
\(H_\varepsilon(r)\) in $(0, +\infty)$
to prove the following nodal geometry of bound states.

\begin{lem}
 \label{lem-center-maximum}
For any fixed $\varepsilon\in(0,\varepsilon_*)$, suppose that $1<\beta_\varepsilon
  <b_\varepsilon\leq+\infty$ are as in \eqref{A:eq-root-limits}, and
$u(r)$ is a solution of \eqref{eq-radial}
satisfying $u(0)>0$.
 Then
 \begin{equation}
  \label{eq-center-range}
  \beta_\varepsilon<u(0)<b_\varepsilon,
 \end{equation}
 and  $r=0$ is the unique global maximum point of $\abs {u(r)}$ in $[0,+\infty)$, in the sense that
\begin{equation}
\label{eq-global-maximum}
\norm{u}_{L^\infty(\R^N)}
 =u(0)>\abs{u(r)},
 \ \ r>0.
\end{equation}
 Moreover, $u'(r)$ has exactly one zero point $r_I$ in each
connected component $I$ of $\{r\ge0:u(r)\neq0\}$, where
\begin{equation}\label{urI-beta}
 \abs{u(r_I)}
 =\max_{r\in I}\abs{u(r)}>\beta_\varepsilon.
\end{equation}
\end{lem}

\begin{proof}
Define
\[
 \alpha:=u(0)>0.
\]
It then gives from \eqref{eq-defineH} and \eqref{eq-H-positive} that
\[
 F_\varepsilon(\alpha)=H_\varepsilon(0)>0,
\]
which implies that $\alpha>\beta_\varepsilon$ holds in view of Lemma~\ref{lem-scalar-structure}.
Multiplying   \eqref{eq-radial}
by $r^{N-1}$ and integrating from $0$ to $r$, we obtain that
\begin{equation}
 \label{eq-integral-radial0}
 r^{N-1}u'(r)
 =-\int_0^r s^{N-1}f_\varepsilon(u(s))\,ds,
 \ \   r>0.
\end{equation}
Since the proof of Lemma~\ref{lem-scalar-structure} gives that $b_\varepsilon=+\infty$ holds for the case $q<p$,
in order to  prove \eqref{eq-center-range}, it next suffices to prove $\alpha<b_\varepsilon$ for the case $q>p$.

We now suppose $q>p$.
If $\alpha=b_\varepsilon$, then it yields from Lemma~\ref{lem-scalar-structure} that
$f_\varepsilon(b_\varepsilon)=0$.
Hence, the constant function $u(r)\equiv b_\varepsilon$
is a solution of   \eqref{eq-radial}
satisfying
$u(0)=\alpha$ and $u'(0)=0$.
By the ODE uniqueness theorem, we obtain that
$u(r)\equiv b_\varepsilon$  holds for any $r\ge0$,
which is impossible. If $\alpha>b_\varepsilon$, then the decaying property of
\eqref{eq-radial} implies that there exists $R_1>0$
satisfying $u(R_1)<b_\varepsilon$.
Since $u(0)=\alpha>b_\varepsilon$, the continuity of
$u(r)$ gives that there exists at least one point in $(0,R_1)$
such that $u(r)=b_\varepsilon$.
Since $u(r)$ is continuous on $[0,R_1]$ and
$u(0)>b_\varepsilon>u(R_1)$, we define
\[
 R:=\min\{r\in[0,R_1]:\, u(r)=b_\varepsilon\}.
\]
It then follows that $0<R<R_1$ and
\[
 u(R)=b_\varepsilon,
 \ \ u(r)>b_\varepsilon
 \ \ \text{for any}\ \ 0\le r<R.
\]
Note from \eqref{eq-f-sign-high} that
$f_\varepsilon(u(r))<0$ holds for any $0\le r<R$.
We then deduce from \eqref{eq-integral-radial0} that
$u'(r)>0$ holds for any $0<r\le R$.
Consequently,
\[
 b_\varepsilon=u(R)
 =\alpha+\int_0^R u'(r)\,dr
 >\alpha>b_\varepsilon,
\]
which is a contradiction. We thus have $\alpha<b_\varepsilon$ for the case $q>p$, which
 therefore proves \eqref{eq-center-range}.

By  \eqref{eq-f-sign-high}, \eqref{eq-center-range}
and \eqref{eq-integral-radial0}, we obtain that
\[
 u''(0)
 =\lim_{r\searrow0}\frac{u'(r)}r
 =-\frac1N f_\varepsilon(\alpha)<0.
\]
Hence, the origin $r=0$ is a strict local maximum point of $u(r)$.
Moreover, it yields from \eqref{eq-defineH} and
\eqref{eq-Hamiltonian-derivative} that
\[
 H_\varepsilon(r)
 =F_\varepsilon(\alpha)
  -(N-1)\int_0^r\frac{\abs{u'(s)}^2}{s}\,ds
 <F_\varepsilon(\alpha),
 \ \   r>0,
\]
where the strict inequality follows from the fact that $u''(0)<0$.
If $\abs{u(r)}=\alpha$ holds for some $r>0$, then
we deduce from \eqref{eq-fF} and \eqref{eq-defineH} that
$H_\varepsilon(r)\ge F_\varepsilon(\alpha)$,
which is however a contradiction.
Since $\abs{u(r)}<\alpha$ holds for any sufficiently small
$r>0$,   the continuity of $u(r)$ gives that
\[
 \abs{u(r)}<\alpha
 \ \ \text{for any}\ \, r>0.
\]
Therefore, $r=0$ is the unique global maximum point
of $\abs{u(r)}$ in $[0,+\infty)$, and hence \eqref{eq-global-maximum} holds.

Define
$y(r):=\abs{u(r)}$ for any $r\in I$, where the set $I$ is any connected component of
$\{r\ge0:\,u(r)\neq0\}$.
Since $f_\varepsilon(s)$ is odd in $\R$ and $u(r)$ does not change the sign on $I$, 
we deduce from
\eqref{eq-radial} that
\[
 y''(r)+\frac{N-1}{r}y'(r)+f_\varepsilon(y(r))=0,
 \ \ r\in I\cap(0,+\infty).
\]
Moreover, $y'(r)=0$ if and only if $u'(r)=0$ holds
for any $r\in I$.
For any $r_0\in I\cap(0,+\infty)$ satisfying
$u'(r_0)=0$, we then deduce from \eqref{eq-fF},
\eqref{eq-defineH} and \eqref{eq-H-positive} that
\[
 F_\varepsilon(y(r_0))
 =F_\varepsilon(u(r_0))
 =H_\varepsilon(r_0)>0.
\]
Applying Lemma~\ref{lem-scalar-structure}, we thus derive from \eqref{eq-center-range} and \eqref{eq-global-maximum}
that
\begin{equation}
 \label{eq-critical-point-height}
 \beta_\varepsilon<y(r_0)\le\alpha<b_\varepsilon,
 \ \ f_\varepsilon(y(r_0))>0.
\end{equation}
Since $y'(r_0)=0$, it follows from \eqref{eq-radial} that
\begin{equation}
 \label{eq-allmaximum}
 y''(r_0)=-f_\varepsilon(y(r_0))<0.
\end{equation}
Thus, $r_0$ is a strict local maximum point of $y(r)$,
where $y(r_0)>\beta_\varepsilon$.
If $0\in I$, the same conclusion holds at $r=0$,
since $y(0)=\alpha>\beta_\varepsilon$
and $y''(0)=u''(0)<0$.

We  next claim that $u'(r)$ has at most one zero point in $I$.
On the contrary, suppose
$u'(r_1)=u'(r_2)=0$ holds for some $r_1,r_2\in I$
satisfying $r_1<r_2$.
The preceding arguments show that both $r_1$ and $r_2$
are strict local maximum points of $y(r)$ in $[0, +\infty)$,
including the case $r_1=0$.
Hence, $y(r)$ attains its minimum on $[r_1,r_2]$
at some $r_3\in(r_1,r_2)$, where
\[
 y'(r_3)=0,
 \ \ y''(r_3)\ge0.
\]
This however contradicts \eqref{eq-allmaximum}, and the above claim is proved. 
It remains to prove that $u'(r)$ admits  zero points in $I$.
We deduce from \eqref{eq-radial} that if $0\in I$, then $r=0$ is a zero point of $u'(r)$.
Suppose now that $0\notin I$.
We then have $y(r)>0$ for any $r\in I$.
By the definition of $I$ and the continuity of $u(r)$,
we get that
\[
 \lim_{\substack{r\to a\\ r\in I}}y(r)=0
\]
holds for any finite endpoint $a$ of $I$.
If $I$ is unbounded, then it follows from   \eqref{eq-radial} that $y(r)\to0$
as $r\to+\infty$.
Thus, $y(r)$ attains a positive maximum
at an interior point of $I$, where $u'(r)=0$. We therefore conclude that $u'(r)$ has exactly one zero point in $I$.

Finally, we denote $r_I$ the unique zero point of $u'(r)$ in $I$, where $r_I=0$ if $0\in I$.
We then deduce from \eqref{eq-center-range}
and \eqref{eq-critical-point-height} that
\[
 \abs{u(r_I)}
 =\max_{r\in I}\abs{u(r)}>\beta_\varepsilon.
\]
This proves \eqref{urI-beta}, which completes the proof of Lemma~\ref{lem-center-maximum}. 
\end{proof}

\subsection{Proof of Proposition~\ref{prop-amplitude}}

The purpose of this subsection is to establish Proposition 2.1, for which we first establish the following oscillation lemma.

\begin{lem}
 \label{lem-blowup-oscillation}
For $N\ge3$, $1<p<\frac{N+2}{N-2}$ and $q>1$,
suppose $w(r)$ satisfies
\begin{equation}
 \label{eq-tau-limit-equation}
 \begin{cases}
  w''(r)+\frac{N-1}{r}w'(r)
  +\abs{w(r)}^{p-1}w(r)
  -\tau\abs{w(r)}^{q-1}w(r)=0
  \,\ \ \text{in}\ \, r>0,\\
  w(0)=1,
  \ \ w'(0)=0,
  \ \ \abs{w(r)}\le1 \ \ \text{for any}\ \  r\ge0,
 \end{cases}
\end{equation}
where either $\tau=0$, or $0<\tau<1$ and $q>p$. 
Then $w(r)$ has infinitely many zero points
in $(0,+\infty)$.
\end{lem}

\begin{proof}
Define
\[
 G_\tau(s)
 :=\frac{\abs{s}^{p+1}}{p+1}
 -\frac{\tau\abs{s}^{q+1}}{q+1},
 \ \ s\in\mathbb R,
\]
and
\begin{equation}
 \label{eq-tau-pohozaev}
 \mathcal P_\tau(r)
 :=r^N\left[\frac12\abs{w'(r)}^2+G_\tau(w(r))\right]
 +\frac{N-2}{2}r^{N-1}w(r)w'(r),
 \ \ r\ge0.
\end{equation}
It then follows from \eqref{eq-tau-limit-equation} and
\eqref{eq-tau-pohozaev} that
\begin{equation}
 \label{eq-tau-pohozaev-derivative}
 \mathcal P_\tau'(r)
 =r^{N-1}\left[
 A_p\abs{w(r)}^{p+1}
 +\tau B_q\abs{w(r)}^{q+1}
 \right],
 \ \ r>0,
\end{equation}
where
\begin{equation}
 \label{eq-AqBp}
 A_p:=\frac{N}{p+1}-\frac{N-2}{2}>0,
 \ \ B_q:=\frac{N-2}{2}-\frac{N}{q+1}.
\end{equation}
We first prove that for either $\tau=0$ or
$0<\tau<1$ and $q>p$,
\begin{equation}
 \label{eq-P-positive}
 \mathcal P_\tau'(r)\ge0,
 \ \ \mathcal P_\tau(r)>0
 \ \ \text{for any}\ \ r>0.
\end{equation}
For the case $\tau=0$, it directly follows from
\eqref{eq-tau-pohozaev-derivative} that
\begin{equation}
\label{eq-P0-derivative-bound}
\mathcal P_0'(r)
 =A_pr^{N-1}\abs{w(r)}^{p+1}\ge0,
 \ \ r>0.
\end{equation}
As for the case $0<\tau<1$ and $q>p$, we deduce from
\eqref{eq-AqBp} that
\begin{align}
\label{ApBq}
A_p+B_q
  =N\left(\frac1{p+1}-\frac1{q+1}\right)>0.
\end{align}
Moreover, since $q>p$ and $\abs{w(r)}\le1$, we have
\[
 \abs{w(r)}^{q+1}
 \le \abs{w(r)}^{p+1},
 \ \ r>0.
\]
It then follows from \eqref{eq-tau-pohozaev-derivative} and \eqref{ApBq} that
\begin{equation}
\label{eq-Ptau-derivative-bound}
\begin{aligned}
  \mathcal P_\tau'(r)
  &\ge A_pr^{N-1}\left[
  \abs{w(r)}^{p+1}-\tau\abs{w(r)}^{q+1}
  \right]\\
  &\ge(1-\tau)A_pr^{N-1}\abs{w(r)}^{p+1}\ge0, \ \ r>0. 
 \end{aligned}
\end{equation}
We thus obtain that for either case of $\tau$, 
\begin{equation}
\label{P'>0}
\mathcal P_\tau'(r)\ge0, \ \ r>0. 
\end{equation}
Since $w(0)=1$ and $w(r)$ is continuous on $[0,+\infty)$,
we deduce from \eqref{eq-P0-derivative-bound}
and \eqref{eq-Ptau-derivative-bound} that
$\mathcal P_\tau'(r)>0$ holds
for any sufficiently small $r>0$.
Noting that
$\mathcal P_\tau(0)=0$, 
we then obtain from \eqref{P'>0} that for either case of $\tau$, 
\[
 \mathcal P_\tau(r)
 =\int_0^r\mathcal P_\tau'(s)\,ds>0,
 \ \ r>0,
\]
which proves \eqref{eq-P-positive}.

To prove Lemma~\ref{lem-blowup-oscillation}, 
on the contrary, suppose that $w(r)$ has only finitely
many zero points in $(0,+\infty)$.
Then there exists $R_0>0$ satisfying
$w(r)\neq0$ for any $r\ge R_0$.
Define for $r\ge R_0$, 
\[
 y(r):=\abs{w(r)},
 \ \ A(r):=r^{N-1}y'(r),
 \ \ r\ge R_0.
\]
For either case of $\tau$, we have
\[
 (1-\tau)s^p\le s^p-\tau s^q\le s^p
 \ \ \text{for any}\ \ s\in(0,1].
\]
Since $w(r)$ does not change the sign on $[R_0,+\infty)$,
we deduce from \eqref{eq-tau-limit-equation} that
\begin{equation}
 \label{eq-dAr}
 A'(r)
 =-r^{N-1}[y(r)^p-\tau y(r)^q]
 \le-(1-\tau)r^{N-1}y(r)^p<0,\ \ r\ge R_0.
\end{equation}
We claim that $A(R_1)<0$ holds for some $R_1\ge R_0$.
On the contrary, suppose $A(r)\ge0$ holds
for any $r\ge R_0$.
Then $y'(r)\ge0$ and $y(r)\ge y(R_0)>0$ hold
for any $r\ge R_0$.
Integrating \eqref{eq-dAr} over $[R_0,r]$, we thus obtain that
\[
 \begin{aligned}
  A(r)
  &\le A(R_0)
  -(1-\tau)\int_{R_0}^r s^{N-1}y(s)^p\,ds\\
  &\le A(R_0)
  -\frac{(1-\tau)y(R_0)^p}{N}(r^N-R_0^N)
  \to-\infty
  \ \ \text{as}\ \ r\to+\infty,
 \end{aligned}
\]
which however contradicts the assumption that $A(r)\ge0$ holds for any $r\ge R_0$. This proves the claim that $A(R_1)<0$ holds for some $R_1\ge R_0$.
We then deduce from \eqref{eq-dAr} that
$A(r)\le A(R_1)<0$ holds for any $r\ge R_1$.
Therefore, $y(r)$ is strictly decreasing
on $[R_1,+\infty)$.

Define
\begin{equation}
 \label{eq-define}
 c_0:=\frac{(1-\tau)(1-2^{-N})}{N}>0.
\end{equation}
For any $r\ge2R_1$, integrating \eqref{eq-dAr}
over $[r/2,r]$, we deduce that
\[
 \begin{aligned}
  -r^{N-1}y'(r)
  &=-A(r/2)
  +\int_{r/2}^r s^{N-1}
  [y(s)^p-\tau y(s)^q]\,ds\\
  &\ge(1-\tau)y(r)^p
  \int_{r/2}^r s^{N-1}\,ds
  =c_0r^Ny(r)^p,
 \end{aligned}
\]
where we have used $A(r/2)<0$, together with the monotonicity
of $y(r)$ on $[R_1,+\infty)$.
It then follows that
\[
 [y(r)^{1-p}]'
 =-(p-1)y(r)^{-p}y'(r)
 \ge(p-1)c_0r,\ \ r\ge2R_1.
\]
Integrating the above inequality over $[2R_1,r]$, we obtain that
\[
 y(r)^{1-p}
 \ge y(2R_1)^{1-p}
 +\frac{(p-1)c_0}{2}(r^2-4R_1^2),
\]
which implies that
\begin{equation}
 \label{eq-slow-decay}
 y(r)\le Cr^{-\frac2{p-1}}, 
 \ \ r\ge4R_1,
\end{equation}
where $C>0$ is independent of $r>0$.
For any $r\ge R_1$, we deduce from \eqref{eq-dAr} that
\[
 \begin{aligned}
  y(r)
  &\ge y(r)-y(2r) =-\int_r^{2r}A(s)s^{1-N}\,ds\\
  &\ge[-A(r)]\int_r^{2r}s^{1-N}\,ds=\frac{1-2^{2-N}}{N-2}\,r\abs{y'(r)}.
 \end{aligned}
\]
Applying this estimate, it thus yields from \eqref{eq-slow-decay} that
\begin{equation}
 \label{eq-slow-gradient}
 \abs{y'(r)}
 \le\frac{N-2}{1-2^{2-N}}\frac{y(r)}r
 \le Cr^{-1-\frac2{p-1}},\ \ r\ge4R_1,
\end{equation}
where $C>0$ is independent of $r>0$.

By the definition of $G_\tau(s)$, we have
\[
 0\le G_\tau(w(r))
 =G_\tau(y(r))
 \le\frac{y(r)^{p+1}}{p+1},\ \ r\ge R_0.
\]
Substituting \eqref{eq-slow-decay} and
\eqref{eq-slow-gradient} into \eqref{eq-tau-pohozaev} gives that
\[
 \begin{aligned}
  \abs{\mathcal P_\tau(r)}
  &\le C\left[
  r^N\abs{y'(r)}^2
  +r^Ny(r)^{p+1}
  +r^{N-1}y(r)\abs{y'(r)}
  \right]\\
  &\le Cr^{N-2-\frac4{p-1}}
  \to0
  \ \ \text{as}\ \ r\to+\infty,
 \end{aligned}
\]
since $p<\frac{N+2}{N-2}$.
On the other hand, we derive from \eqref{eq-P-positive} that
\[
 \mathcal P_\tau(r)\ge\mathcal P_\tau(1)>0
 \ \ \text{for any}\ \, r\ge1,
\]
 a contradiction.
This therefore completes the proof of Lemma \ref{lem-blowup-oscillation}.
\end{proof}

Applying
Lemma~\ref{lem-blowup-oscillation}, we now prove Proposition \ref{prop-amplitude}.

\begin{proof}[Proof of Proposition \ref{prop-amplitude}]
 In view of Lemma~\ref{lem-center-maximum}, it suffices  
 to prove that $u_\varepsilon(0)$
is uniformly bounded for any sufficiently small
$\varepsilon>0$.
On the contrary, suppose that there exist
$\varepsilon_n\in(0,\varepsilon_*)$, where $\varepsilon_n \searrow 0$ as $n \to \infty$, and a bound state
$u_n:=u_{\varepsilon_n}\in\mathcal B_{\varepsilon_n,k}$
of \eqref{mainequation} satisfying
\[
 \ \ \alpha_n:=u_n(0)\to\infty
 \ \ \text{as}\ \ n\to\infty.
\]
Define
\begin{equation}
 \label{eq-blowup-scaling}
 \begin{aligned}
  w_n(r)
  &:=\frac1{\alpha_n}
  u_n\left(\alpha_n^{-\frac{p-1}{2}}r\right),
  \ \ r\ge0,\\
  \delta_n^2
  &:=\alpha_n^{1-p}>0, \ \
    \tau_n
  :=\varepsilon_n\alpha_n^{q-p}>0,
 \end{aligned}
\end{equation}
so that $\delta_n^2\to0$ as $n\to\infty$. Applying Lemma~\ref{lem-center-maximum},
it then follows from \eqref{eq-radial} and \eqref{eq-blowup-scaling} that
$w_n(r)$ is a solution of
\begin{equation}
\begin{split}
 \label{eq-blowup-equation}
 &w_n''(r)+\frac{N-1}{r}w_n'(r)
 -\delta_n^2w_n(r)
 \\
 +&\abs{w_n(r)}^{p-1}w_n(r)
 -\tau_n\abs{w_n(r)}^{q-1}w_n(r)=0,
 \ \ r>0,
 \end{split}
\end{equation}
satisfying
\begin{equation}
 \label{eq-blowup-data}
 w_n(0)=1,
 \ \ w_n'(0)=0,
 \ \ \abs{w_n(r)}\le1
 \ \ \text{for any}\ \, r\ge0.
\end{equation}
Moreover, $w_n(r)$ has exactly $k$ zero points
in $(0,+\infty)$, and
\begin{equation}
 \label{eq-blowup-decay}
 w_n(r)\to0
 \ \ \text{as}\ \ r\to+\infty.
\end{equation}

By Lemmas~\ref{lem-scalar-structure} and
\ref{lem-center-maximum}, we have
$f_{\varepsilon_n}(\alpha_n)>0$.
It then follows from \eqref{eq-fF} and
\eqref{eq-blowup-scaling} that
\[
 0<
 \frac{f_{\varepsilon_n}(\alpha_n)}{\alpha_n^p}
 =1-\delta_n^2-\tau_n.
\]
Consequently,
\begin{equation}
 \label{eq-tau-range}
 0<\tau_n<1-\delta_n^2<1.
\end{equation}
Up to a subsequence if necessary, we may assume that
\begin{equation}
 \label{2:eq-tau-range}
 \tau_n\to\tau\in[0,1]
 \ \ \text{as}\ \ n\to\infty.
\end{equation}
By the continuous dependence of solutions for
\eqref{eq-blowup-equation} and \eqref{eq-blowup-data}
on the parameters $\delta_n^2$ and $\tau_n$, we then obtain that
\begin{equation}
 \label{eq-blowup-local-convergence}
 w_n\to w
 \,\ \ \text{in}\ \ C^2_{\mathrm{loc}}\bigl([0,+\infty)\bigr)
 \ \ \text{as}\ \ n\to\infty,
\end{equation}
where $w(r)$ solves \eqref{eq-tau-limit-equation}
with $\tau=\lim_{n\to\infty}\tau_n\in[0,1]$.
In order to finish the proof, in view of (\ref{2:eq-tau-range}), we next need to derive a contradiction by considering separately the following three cases:

\textit{Case 1: $q<p$.}
For this case, we obtain from  \eqref{eq-blowup-scaling} that
\[
 \tau_n
 =\frac{\varepsilon_n}{\alpha_n^{p-q}}\to0
 \ \ \text{as}\ \ n\to\infty,
\]
and hence $\tau=0$. It then yields from
Lemma~\ref{lem-blowup-oscillation} that
$w(r)$ has infinitely many zero points in $(0,+\infty)$.
Applying \eqref{eq-tau-limit-equation} and
the ODE uniqueness theorem, we then conclude that all these zero points are simple.
We can therefore choose $k+1$ pairwise disjoint intervals
$[a_j,b_j]\subset(0,+\infty)$ such that
\[
 w(a_j)w(b_j)<0,
 \ \ j=1,\cdots,k+1.
\]
We then deduce from \eqref{eq-blowup-local-convergence} that
\[
 w_n(a_j)w_n(b_j)<0,
 \ \ j=1,\cdots,k+1,
\]
where $n\in \mathbb{N}$ is sufficiently large.
The intermediate value theorem then yields that
$w_n(r)$ has at least one zero point in each interval
$(a_j,b_j)$.
Thus, $w_n(r)$ has at least $k+1$ distinct zero points
in $(0,+\infty)$, which however contradicts the fact that
$w_n(r)$ has exactly $k$ zero points in $(0,+\infty)$.
This proves \eqref{eq-amplitude-bound-qgeqp} for Case 1.

\textit{Case 2: $q>p$ and $0\le \tau<1$.}
For this case,  by Lemma~\ref{lem-blowup-oscillation} and
the ODE uniqueness theorem, $w(r)$ has infinitely many
simple zero points in $(0,+\infty)$.
Repeating the argument as above Case 1 with the first
$k+1$ zero points of $w(r)$, we then deduce from
\eqref{eq-blowup-local-convergence} that for any sufficiently large $n>0$, $w_n(r)$
has at least $k+1$ zero points in $(0,+\infty)$, 
which is again a contradiction. 
This proves \eqref{eq-amplitude-bound-qgeqp} for Case 2.

\textit{Case 3: $q>p$ and $\tau=1$.}
For this case, since the constant function $w(r)\equiv1$ satisfies
\eqref{eq-tau-limit-equation} with $\tau=1$,
 the ODE uniqueness theorem yields that
$w(r)\equiv 1$ holds for any $r\ge0$.
It then follows from \eqref{eq-blowup-local-convergence} that
\begin{equation}
 \label{eq-plateau-local}
 w_n\to1
 \,\ \ \text{in}\ \ C^2_{\mathrm{loc}}\bigl([0,+\infty)\bigr)
 \ \ \text{as}\ \ n\to\infty. 
\end{equation}
Define  
\begin{equation}
 \label{eq-fnFn}
 \begin{aligned}
  f_n(s)
  &:=-\delta_n^2s+\abs{s}^{p-1}s
  -\tau_n\abs{s}^{q-1}s,\\
  F_n(s)
  &:=-\frac{\delta_n^2s^2}{2}
  +\frac{\abs{s}^{p+1}}{p+1}
  -\frac{\tau_n\abs{s}^{q+1}}{q+1},
 \end{aligned}
\end{equation}
where $0<\delta_n^2<1$ and $0<\tau_n<1$ are defined by
\eqref{eq-blowup-scaling}.
We also define the energy function of
\eqref{eq-blowup-equation} by
\begin{equation}
 \label{eq-En}
 E_n(r):=\frac12\abs{w_n'(r)}^2+F_n(w_n(r)),
 \ \ r\ge0.
\end{equation}
By \eqref{eq-defineH} and \eqref{eq-blowup-scaling}, we then deduce from
\eqref{eq-fnFn} and \eqref{eq-En} that
\begin{equation}
 \label{eq-energy-scaling}
 E_n(r)
 =\alpha_n^{-(p+1)}
 H_{\varepsilon_n}
 \left(\alpha_n^{-\frac{p-1}{2}}r\right),
 \ \ r\ge0,
\end{equation}
where $H_{\varepsilon_n}(r)$ is the energy function
associated with $u_n(r)$.
It then follows from \eqref{eq-Hamiltonian-derivative},
\eqref{eq-H-positive} and \eqref{eq-energy-scaling}
that
\begin{equation}
 \label{eq-En-properties}
 \begin{aligned}
  E_n'(r)
  &=-\frac{N-1}{r}\abs{w_n'(r)}^2\le0
  \ \ \text{for any}\ \ r>0,\\
  E_n(r)&>0
  \ \ \text{for any}\ \ r\ge0,
  \ \ \lim_{r\to+\infty}E_n(r)=0.
 \end{aligned}
\end{equation}

For any fixed $\vartheta\in(0,1)$, we now define
\begin{equation}
 \label{eq-departure-radius}
 R_n:=\inf\{r>0:\, w_n(r)=\vartheta\}.
\end{equation}
By the continuity of $w_n(r)$, we then derive from \eqref{eq-blowup-data} and \eqref{eq-blowup-decay} that
\[
 0<R_n<+\infty,
 \ \ w_n(R_n)=\vartheta,
 \ \ w_n(r)>\vartheta
 \ \ \text{for any}\ \ 0\le r<R_n.
\]
It then yields from \eqref{eq-plateau-local} that for fixed $R>0$ and any sufficiently large $n>0$, 
\[
 \min_{0\le r\le R}w_n(r)>\vartheta, 
\]
and hence $R_n>R$.
Since $R>0$ is arbitrary, we conclude that
\begin{equation}
 \label{eq-departure-radius-limit}
 R_n\to+\infty
 \ \ \text{as}\ \ n\to\infty.
\end{equation}
Since $q>p$, $0<\vartheta<1$,
$\delta_n^2\to0$ and $\tau_n\to1$ as $n\to\infty$,
we deduce from \eqref{eq-fnFn} that
\[
 F_n(\vartheta)\to
 \frac{\vartheta^{p+1}}{p+1}
 -\frac{\vartheta^{q+1}}{q+1}>0
 \ \ \text{as}\ \ n\to\infty.
\]
It then follows from \eqref{eq-En} and
\eqref{eq-departure-radius} that for any sufficiently large $n>0$,
\begin{equation}
 \label{eq-energy-departure}
 E_n(R_n)\ge F_n(\vartheta)\ge c_1>0, 
\end{equation}
where
\[
 c_1:=\frac12\left[
 \frac{\vartheta^{p+1}}{p+1}
 -\frac{\vartheta^{q+1}}{q+1}
 \right]>0
\]
is independent of $n>0$.

On the other hand, we deduce from
\eqref{eq-blowup-data}, \eqref{eq-fnFn},
\eqref{eq-En} and \eqref{eq-En-properties} that
\[
  \frac12\abs{w_n'(r)}^2
   \le F_n(1)-F_n(w_n(r)) 
   \le\frac{1-\abs{w_n(r)}^{p+1}}{p+1}
  \le\frac1{p+1},\ \ r\ge 0.
\]
Consequently,
\begin{equation}
 \label{eq-wn-derivative-bound}
 \abs{w_n'(r)}\le L,\ \ r\ge0,
\end{equation}
where $L:=\left(\frac2{p+1}\right)^{\frac12}>0$
is independent of $n>0$ and $r>0$.
Suppose $R_n\le a<b\le+\infty$ and
$w_n(r)$ is monotone on $[a,b)$.
We then deduce from \eqref{eq-blowup-data},
\eqref{eq-En-properties}
and \eqref{eq-wn-derivative-bound} that
\begin{equation}
\label{eq-half-cycle-loss}
\begin{aligned}
  E_n(a)-E_n(b)
  &=(N-1)\int_a^b\frac{\abs{w_n'(r)}^2}{r}\,dr\\
  &\le\frac{(N-1)L}{R_n}
  \int_a^b\abs{w_n'(r)}\,dr\\
  &\le\frac{2(N-1)L}{R_n}=:\rho_n,
 \end{aligned}
\end{equation}
where $\rho_n\to0$ as $n\to\infty$
in view of \eqref{eq-departure-radius-limit}.
Here, if $b=+\infty$, then we define
$E_n(b):=\lim\limits_{r\to+\infty}E_n(r)=0$.

It follows from \eqref{eq-fnFn}, \eqref{eq-En} and
\eqref{eq-En-properties} that
\[
 \frac12\abs{w_n'(r)}^2=E_n(r)>0
\]
holds at any zero point $r>0$ of $w_n(r)$.
Thus, $w_n(r)$ and $w_n'(r)$ have no common zero point.
It then yields  from Lemma~\ref{lem-center-maximum} and
\eqref{eq-blowup-scaling} that $w_n'(r)$ has exactly
$k$ zero points in $(0,+\infty)$.
Define $\xi_{0,n}:=R_n$, and denote
$\xi_{1,n}<\cdots<\xi_{k,n}$ the positive
zero points of $w_n'(r)$. 
Since $R_n$ belongs to the connected component of
$\{r\ge0:\, w_n(r)\neq0\}$ containing the origin,
it implies from Lemma~\ref{lem-center-maximum} that
all these positive zero points belong to $(R_n,+\infty)$.
Therefore, $w_n(r)$ is monotone on  $[\xi_{k,n},+\infty)$ and
$[\xi_{j,n},\xi_{j+1,n}]$ for each integer $j$
satisfying $0\le j<k$.
We then obtain from \eqref{eq-energy-departure} and
\eqref{eq-half-cycle-loss} that
\begin{equation}
 \label{lossenergy}
 \begin{aligned}
  0<c_1
  \le E_n(R_n)
  &=\sum_{j=0}^{k-1}
  [E_n(\xi_{j,n})-E_n(\xi_{j+1,n})]
  +E_n(\xi_{k,n})\\
  &\le(k+1)\rho_n\to0
  \ \ \text{as}\ \ n\to\infty.
 \end{aligned}
\end{equation}
This is a contradiction, which therefore proves \eqref{eq-amplitude-bound-qgeqp} for Case 3.
This completes the proof of
Proposition~\ref{prop-amplitude}.
\end{proof}

\section{Convergence of bound states as
\texorpdfstring{$\varepsilon\searrow0$}{epsilon decreases to 0}}
\label{sec-noescape}

The purpose of this section is to address the proof of Theorem~\ref{thm-main}.
 We shall first prove that all zero points
of the $k$-node bound state $u_\varepsilon(r)\in\mathcal B_{\varepsilon,k}$ for \eqref{mainequation} remain bounded uniformly
as $\varepsilon\searrow0$.
We then establish the uniform $H^1(\R^N)$ bound
of the $k$-node bound state $u_\varepsilon(r)$ and its uniform
convergence to $Q_k(r)$ in
$H^1(\R^N)\cap L^\infty(\R^N)$
as $\varepsilon\searrow0$,
where $Q_k(r)$ is the unique  $k$-node bound state
of \eqref{eq-limit} satisfying $Q_k(0)>0$.

For each fixed integer $k\ge1$, suppose that
$u_n\in\mathcal B_{\varepsilon_n,k}$ is a $k$-node bound state of \eqref{mainequation}, 
where $\varepsilon_n\in(0,\varepsilon_*)$ satisfies
$\varepsilon_n\searrow0$ as $n\to \infty$. 
It follows from Proposition~\ref{prop-amplitude}
that the sequence $\{u_n\}$ is uniformly bounded
in $L^\infty(\R^N)$ as $n \to \infty$. 
Up to a subsequence if necessary, then we may assume that
there exists $\alpha_* \geq 0$ such that
\begin{equation}
 \label{eq-alpha-limit}
 \alpha_n:=u_n(0)\to\alpha_*
 \ \ \text{as}\ \ n\to\infty.
\end{equation}
Applying the continuous dependence of solutions for \eqref{eq-radial}
on the initial value $u_n(0)=\alpha_n$
and the parameter $\varepsilon_n$,
 we deduce from
\eqref{eq-alpha-limit} that
\begin{equation}
 \label{eq-local-limit}
 u_n \to V
 \,\ \ \text{in}\ \ C^2_{\mathrm{loc}}\bigl([0,+\infty)\bigr)
 \ \ \text{as}\ \ n\to\infty,
\end{equation}
where $V=V(r)$ satisfies the following initial value problem
\[
 \begin{cases}
  V''+\frac{N-1}{r}V'-V+\abs{V}^{p-1}V=0
  \,\ \ \text{in}\ \ r>0,\\
  V(0)=\alpha_*,
  \ \ V'(0)=0.
 \end{cases}
\]

Recall that
$Q_0(r)$ denotes
the unique positive solution of \eqref{eq-limit} centered at the origin.
We first identify the limit function $V(r)$ of \eqref{eq-local-limit} in the following lemma.

\begin{lem}
 \label{lem-identification}
 For each integer $k\ge1$, suppose
 $u_n\in\mathcal B_{\varepsilon_n,k}$ is a $k$-node bound state of \eqref{mainequation},
 where $\varepsilon_n\in(0,\varepsilon_*)$ satisfies
 $\varepsilon_n\searrow0$ as $n\to\infty$.
 Then the limit function $V(r)$ in \eqref{eq-local-limit}
 satisfies
 \begin{equation}
  \label{eq-limit-Qm}
  V(r)\equiv Q_m(r), 
  \ \ r\ge0,
 \end{equation}
 where the integer $m$ satisfies $0\le m\le k$,  and $Q_m$ is an $m$-node bound state of \eqref{eq-limit}. 
\end{lem}

\begin{proof}
We deduce from \eqref{eq-root-limits},
\eqref{eq-center-range} and \eqref{eq-alpha-limit} that
\[
 \alpha_*
 =\lim_{n\to\infty}\alpha_n
 \ge\lim_{n\to\infty}\beta_{\varepsilon_n}
 =\beta_0=\left(\frac{p+1}{2}\right)^{\frac1{p-1}}>1.
\]
Since $V(0)=\alpha_*>1$, 
$V(r)$ is not a constant in $[0,+\infty)$.
Setting $n\to\infty$ for \eqref{eq-H-positive},
it then follows from \eqref{eq-fF}, \eqref{eq-defineH}
and \eqref{eq-local-limit} that
\begin{equation}
 \label{eq-limit-H-positive}
 H_{0,V}(r)
 :=\frac12\abs{V'(r)}^2+F_0(V(r))\ge0, 
\ \ r\ge0.
\end{equation}
We conclude from
\cite[Theorem~2 and Proposition~2.4(iv)]{tang2026uniqueness}
that either $V(r)\equiv Q_m(r)$
for some integer $m\ge0$,
or $V(r)$ has a fixed sign and oscillates about
$\pm1$ for any sufficiently large $r>0$.
In the latter case, there exists $r_0>0$ satisfying
\[
 V'(r_0)=0,
 \ \ 0<\abs{V(r_0)}<1.
\]
It then yields from \eqref{eq-F0}
and \eqref{eq-limit-H-positive} that for $p>1$, 
\[
 \begin{aligned}
  H_{0,V}(r_0)
  &=F_0(V(r_0))
  =-\frac{\abs{V(r_0)}^2}{2}
  +\frac{\abs{V(r_0)}^{p+1}}{p+1}<0,
 \end{aligned}
\]
which however contradicts \eqref{eq-limit-H-positive}.
This implies that $V(r)\equiv Q_m(r)$ holds
for some integer $m\ge0$.

It remains to prove that $0\leq m\le k$.
On the contrary, suppose $m>k$.
Since $Q_m(r)$ is a nontrivial radial solution
of \eqref{eq-limit}, the ODE uniqueness theorem implies that
any zero point of $Q_m(r)$ in $(0,+\infty)$ is simple.
We can therefore choose $k+1$ pairwise disjoint intervals
$[a_j,b_j]\subset(0,+\infty)$ such that
\[
 Q_m(a_j)Q_m(b_j)<0,
 \ \ j=1,\cdots,k+1.
\]
It then follows from
\eqref{eq-local-limit} that
\[
 u_n(a_j)u_n(b_j)<0,
 \ \ j=1,\cdots,k+1, 
\]
where $n>0$ is sufficiently large. 
The intermediate value theorem yields that
$u_n(r)$ has at least one zero point in each interval
$(a_j,b_j)$.
Thus, 
$u_n(r)$ has at least $k+1$ distinct zero points
in $(0,+\infty)$, which however contradicts the fact that
$u_n\in\mathcal B_{\varepsilon_n,k}$ holds for any sufficiently large $n>0$. 
We therefore conclude that $0 \leq m\le k$.
This proves \eqref{eq-limit-Qm}, which completes
the proof of Lemma~\ref{lem-identification}.
\end{proof}

We next prove that $m=k$ in \eqref{eq-limit-Qm}.
For this purpose, we use a radial Pohozaev identity
of \eqref{eq-radial}.
Define for any fixed $\varepsilon\in(0,\varepsilon_*)$
and any solution $u(r)$ of \eqref{eq-radial},
\begin{equation}
 \label{eq-radial-pohozaev}
 \Pi_{\varepsilon,u}(r)
 :=r^N H_\varepsilon(r)
 +\frac{N-2}{2}r^{N-1}u(r)u'(r),
 \ \ r\ge0,
\end{equation}
where $H_\varepsilon(r)$ is the energy function
associated with $u(r)$ and defined by \eqref{eq-defineH}.
Differentiating \eqref{eq-radial-pohozaev}
and using \eqref{eq-radial}, we then obtain the following
radial Pohozaev identity
\begin{equation}
 \label{eq-radial-pohozaev-derivative}
 \Pi_{\varepsilon,u}'(r)
 =r^{N-1}\mathcal D_\varepsilon\bigl(u(r)\bigr),
 \ \ r>0,\ \ 0<\varepsilon<\varepsilon_*,
\end{equation}
where $\mathcal D_\varepsilon:\mathbb R\to\mathbb R$
is defined as
\begin{equation}
 \label{eq-D-epsilon}
 \mathcal D_\varepsilon(s)
 :=-s^2
 +\left(\frac{N}{p+1}-\frac{N-2}{2}\right)
 \abs{s}^{p+1}
 +\varepsilon
 \left(\frac{N-2}{2}-\frac{N}{q+1}\right)
 \abs{s}^{q+1},
 \ \ s\in\mathbb R.
\end{equation}

\begin{lem}
 \label{lem-small-pohozaev}
 For any $\varepsilon>0$,
 let $\mathcal D_\varepsilon(s)$ be defined by
 \eqref{eq-D-epsilon}, 
 where $N,p,q$ satisfy \eqref{eq-range}.
 Then there exist constants $\eta>0$ and
 $\varepsilon_0>0$, depending only on $N,p,q$,
 such that for any $0<\varepsilon\le\varepsilon_0, $
\begin{equation}
 \label{eq-D-negative}
 \mathcal D_\varepsilon(s)\le-\frac12s^2,
 \ \ \abs{s}\le\eta. 
\end{equation}
\end{lem}

\begin{proof}
Note from \eqref{eq-D-epsilon} that for any $0 <  \varepsilon \leq 1$, 
\[
 \begin{aligned}
  \mathcal D_\varepsilon(s)
  &\le-s^2
  +\left(\frac{N}{p+1}-\frac{N-2}{2}\right)
  \abs{s}^{p+1}
  +\left|\frac{N-2}{2}-\frac{N}{q+1}\right|
  \abs{s}^{q+1}\\
  &\le\left[-1+C\left(
  \abs{s}^{p-1}+\abs{s}^{q-1}
  \right)\right]s^2,\ \  s \in \R, 
 \end{aligned}
\]
where $C>0$ depends only on $N,p,q$.
Taking $\varepsilon_0=1$, we choose $\eta>0$ such that
\[
 C(\eta^{p-1}+\eta^{q-1})\le\frac12.
\]
This thus proves \eqref{eq-D-negative} for all $0 \leq |s| \leq \eta$, which
completes the proof of
Lemma~\ref{lem-small-pohozaev}.
\end{proof}

\begin{lem}
 \label{prop-noescape}
 For each integer $k\ge1$, suppose
 $u_n\in\mathcal B_{\varepsilon_n,k}$ is a $k$-node bound state of \eqref{mainequation},
 where $\varepsilon_n\in(0,\varepsilon_*)$ satisfies
 $\varepsilon_n\searrow0$ as $n\to\infty$.
 Then \eqref{eq-limit-Qm} holds for $m=k$. 
\end{lem}

\begin{proof}
On the contrary, suppose $m<k$ in view of Lemma~\ref{lem-identification}.
Since $Q_m(r)$ has exactly $m$ zero points
in $(0,+\infty)$ and $Q_m(r)\to0$ as $r\to+\infty$,
we can choose $L>0$ such that all these zero points
belong to $(0,L)$ and
\begin{equation}
 \label{eq-L-small}
 \abs{Q_m(L)}<\frac{\eta}{2},
 \ \ Q_m(L)Q_m'(L)<0,
\end{equation}
where $\eta>0$ is given by Lemma~\ref{lem-small-pohozaev}.
Noting that the zero points of $Q_m(r)$
are simple,
we deduce from \eqref{eq-local-limit}
and \eqref{eq-limit-Qm} that
for any sufficiently large $n>0$,
$u_n(r)$ has exactly one zero point in a sufficiently
small neighborhood of each zero point for $Q_m(r)$
 and no other zero point in $[0,L]$. 
Thus, $u_n(r)$ has exactly $m$ zero points in $[0,L]$, where $n>0$ is sufficiently large.
Since $m<k$, we can define for any sufficiently large $n>0$, 
\[
 R_n:=\min\{r>L:\, u_n(r)=0\}>L.
\]
Furthermore, for any fixed $R>L$, 
\[
 \min_{r\in[L,R]}\abs{Q_m(r)}>0.
\]
It then follows from \eqref{eq-local-limit}
and \eqref{eq-limit-Qm} that 
for any sufficiently large $n>0$, 
$u_n(r)$ has no zero point
in $[L,R]$,
and hence $R_n>R$.
Since $R>L$ is arbitrary, we obtain that
\begin{equation}
 \label{eq-Rn-infty}
 R_n\to+\infty
 \ \ \text{as}\ \ n\to\infty.
\end{equation}

We further derive from \eqref{eq-local-limit},
\eqref{eq-limit-Qm} and \eqref{eq-L-small} that for any sufficiently large $n>0$, 
\[
 \abs{u_n(L)}<\eta,
 \ \ u_n(L)u_n'(L)<0. 
\]
Since $u_n(r)\neq 0$ holds for any $r\in[L,R_n)$,
it follows from Lemma~\ref{lem-center-maximum} that
$\abs{u_n(r)}$ is strictly decreasing on $[L,R_n]$.
Consequently,
\begin{equation}
 \label{eq-small-on-LR}
 \abs{u_n(r)}\le\abs{u_n(L)}<\eta, 
\ \ L\le r\le R_n.
\end{equation}
Since $0<\varepsilon_n\le\varepsilon_0$ holds
for any sufficiently large $n>0$,
we deduce from \eqref{eq-radial-pohozaev-derivative},
\eqref{eq-D-negative}
and \eqref{eq-small-on-LR} that
\begin{equation}
 \label{eq-Pi-at-R}
 \Pi_{\varepsilon_n,u_n}(R_n)
 \le\Pi_{\varepsilon_n,u_n}(L)\le C_L,
\end{equation}
where $C_L>0$ is independent of $n>0$.
Here the last inequality follows from
\eqref{eq-fF}, \eqref{eq-defineH},
\eqref{eq-local-limit} and \eqref{eq-radial-pohozaev},
since $L>0$ is fixed.

Applying Lemma~\ref{lem-scalar-structure},
we then deduce from \eqref{eq-fF} and \eqref{eq-F0} that
\[
 F_0(\beta_{\varepsilon_n})
 =\frac{\varepsilon_n\beta_{\varepsilon_n}^{q+1}}{q+1}>0,
\]
which implies that
$\beta_{\varepsilon_n}>\beta_0
=\left(\frac{p+1}{2}\right)^{\frac1{p-1}}>1$.
It thus follows from Lemma~\ref{lem-center-maximum} that for any connected component $J$ of
$\{r\ge0:u_n(r)\neq0\}$, 
\begin{equation}
\label{eq-fixed-bump-height}
\max_{r\in J}\abs{u_n(r)}
 >\beta_{\varepsilon_n}>\beta_0. 
\end{equation}
Choose $\eta_0\in(0,\min\{\eta,\beta_0\})$ such that
\[
 \frac{\eta_0^{p-1}}{p+1}\le\frac14.
\]
It then follows from \eqref{eq-fF} that for any $\varepsilon\in(0,\varepsilon_*)$, 
\begin{equation}
 \label{eq-F-eta-negative}
 F_\varepsilon(\eta_0)
 \le-\frac{\eta_0^2}{2}
 +\frac{\eta_0^{p+1}}{p+1}
 \le-\frac{\eta_0^2}{4}=:-c_2<0. 
\end{equation}
In the rest proof of this lemma, suppose that $n>0$ is sufficiently
large, so that $\varepsilon_n\le\varepsilon_0$
and \eqref{eq-Pi-at-R} holds.
We then derive from \eqref{eq-fixed-bump-height} that
$$\abs{u_n(r)}>\beta_0=\left(\frac{p+1}{2}\right)^{\frac1{p-1}}>\eta_0$$ 
holds for some $r>R_n$.
Since $u_n(R_n)=0$, the continuity of $u_n(r)$
allows us to define
\begin{equation}
 \label{eq-Tn}
 T_n:=\min\{r>R_n:\,  \abs{u_n(r)}=\eta_0\}>R_n.
\end{equation}
It then follows that $\abs{u_n(T_n)}=\eta_0$ and
\[
 \abs{u_n(r)}\le\eta_0<\eta
 \ \ \text{for any}\ \ R_n\le r\le T_n.
\]
Applying \eqref{eq-radial-pohozaev-derivative}
and \eqref{eq-D-negative} on $[R_n,T_n]$,
we deduce from \eqref{eq-Pi-at-R} that
\begin{equation}
 \label{eq-Pi-T-upper}
 \Pi_{\varepsilon_n,u_n}(T_n)
 \le\Pi_{\varepsilon_n,u_n}(R_n)
 \le C_L.
\end{equation}

On the other hand, we deduce from
\eqref{eq-fF}, \eqref{eq-defineH}, \eqref{eq-H-positive},
\eqref{eq-F-eta-negative} and \eqref{eq-Tn} that
\[
 \begin{aligned}
  \frac12\abs{u_n'(T_n)}^2
  &=H_{\varepsilon_n}(T_n)
  -F_{\varepsilon_n}(u_n(T_n))\\
  &\ge-F_{\varepsilon_n}(\eta_0)
  \ge c_2,
 \end{aligned}
\]
where $H_{\varepsilon_n}(r)$ is the energy function
associated with $u_n(r)$. 
Consequently,
\begin{equation}
 \label{eq-slope-at-T}
 \abs{u_n'(T_n)}\ge\sqrt{2c_2}.
\end{equation}
Note from \eqref{eq-Tn}
that $u_n(T_n)u_n'(T_n)\ge0$.
Together with \eqref{eq-slope-at-T}, this yields that
\begin{equation}
 \label{eq-product-positive}
 u_n(T_n)u_n'(T_n)
 =\eta_0\abs{u_n'(T_n)}
 \ge\eta_0\sqrt{2c_2}>0.
\end{equation}
We then derive from \eqref{eq-H-positive},
\eqref{eq-radial-pohozaev} and \eqref{eq-product-positive} that
\begin{equation}
 \label{eq-Pi-T-lower}
 \begin{aligned}
  \Pi_{\varepsilon_n,u_n}(T_n)
  &\ge\frac{N-2}{2}T_n^{N-1}
  u_n(T_n)u_n'(T_n)\\
  &\ge\frac{N-2}{2}\eta_0\sqrt{2c_2}\,T_n^{N-1}
  \to+\infty
  \ \ \text{as}\ \ n\to\infty,
 \end{aligned}
\end{equation}
where the limit follows from \eqref{eq-Rn-infty},
since $T_n>R_n$ and $N\ge3$.
This however contradicts \eqref{eq-Pi-T-upper}.
We thus have $m=k$ in \eqref{eq-limit-Qm},
which completes the proof of
Lemma~\ref{prop-noescape}.
\end{proof}

We next prove the convergence of $u_n$ 
in $H^1(\R^N)\cap L^\infty(\R^N)$ as $n \to \infty$. 

\begin{prop}
 \label{prop-global-compactness}
 For each integer $k\ge1$, suppose that
  $u_n\in\mathcal B_{\varepsilon_n,k}$
 is a $k$-node bound state of \eqref{mainequation}, where $\varepsilon_n\in(0,\varepsilon_*)$ satisfies
 $\varepsilon_n\searrow0$ as $n\to\infty$. 
 Then
 \begin{equation}
  \label{eq-sequential-convergence}
  u_n\to Q_k
  \ \ \text{strongly in}\ \
  H^1(\R^N)\cap L^\infty(\R^N)
  \ \ \text{as}\ \ n\to\infty.
 \end{equation}
\end{prop}

\begin{proof}
It follows from Proposition~\ref{prop-amplitude}
and Lemmas~\ref{lem-identification} and \ref{prop-noescape}
that any subsequence of $\{u_n\}_{n\ge1}$ has a 
subsequence converging to $Q_k$ in
$C^2_{\mathrm{loc}}\bigl([0,+\infty)\bigr)$ as $n \to \infty$,
which thus implies that
\begin{equation}
 \label{eq-Qk-local-convergence}
 u_n\to Q_k
 \,\ \ \text{in}\ \
 C^2_{\mathrm{loc}}\bigl([0,+\infty)\bigr)
 \ \ \text{as}\ \ n\to\infty.
\end{equation}

We now derive a uniform exponential decaying estimate of
$\abs{u_n(r)}+\abs{u_n'(r)}$ 
 for sufficiently large $n>0$.
Choose $\delta\in(0,1)$ such that
\begin{equation}
 \label{eq-delta-choice}
 \delta^{p-1}\le\frac12.
\end{equation}
Since $Q_k(r)$ has exactly $k$ zero points
in $(0,+\infty)$ and $Q_k(r)\to0$ as $r\to+\infty$,
one can choose $R_0>1$ such that all $k$ zero points of $Q_k(r)$
belong to $(0,R_0)$ and
\[
 \abs{Q_k(R_0)}<\frac{\delta}{4},
 \ \ Q_k(R_0)Q_k'(R_0)<0.
\]
By the simplicity of the zero points for $Q_k(r)$,
we deduce from \eqref{eq-Qk-local-convergence} that
all $k$ zero points of $u_n(r)$ belong to $(0,R_0)$
and
\[
 \abs{u_n(R_0)}<\frac{\delta}{2},
 \ \ u_n(R_0)u_n'(R_0)<0, 
\]
where $n>0$ is sufficiently large. 
In the following, we assume that $n>0$ is sufficiently large.
Define 
\[
 y_n(r):=\abs{u_n(r)},
 \ \ r\ge R_0.
\]
Since $u_n(r)$ has no zero point in $[R_0,+\infty)$
and $u_n(R_0)u_n'(R_0)<0$,
we deduce from Lemma~\ref{lem-center-maximum} that
\begin{equation}
 \label{eq-tail-monotonicity}
 y_n'(r)<0,
 \ \ 0<y_n(r)\le y_n(R_0)<\frac{\delta}{2},
 \ \ \forall \, r\ge R_0.
\end{equation}
It then follows from \eqref{eq-radial},
\eqref{eq-delta-choice} and \eqref{eq-tail-monotonicity} that
\begin{equation}
 \label{eq-exterior-inequality}
 \begin{aligned}
 &\quad -y_n''(r)-\frac{N-1}{r}y_n'(r)+\frac12y_n(r) \\
 &=y_n(r)^p-\varepsilon_ny_n(r)^q-\frac12y_n(r)
 \le0,\quad r>R_0.
 \end{aligned}
\end{equation}

Choose some $c\in(0,1/\sqrt2)$ and define
\[
 \Psi(r):=\delta e^{-c(r-R_0)},
 \ \ r\ge R_0.
\]
A direct calculation then gives that
\begin{equation}
 \label{eq-exterior-supersolution}
 -\Psi''(r)-\frac{N-1}{r}\Psi'(r)
 +\frac12\Psi(r)
 =\left[\frac12-c^2+\frac{(N-1)c}{r}\right]
 \Psi(r)\ge0,\ \ r>R_0. 
\end{equation}
Since $y_n(R_0)<\Psi(R_0)$ and
$y_n(r)-\Psi(r)\to0$ as $r\to+\infty$,
applying the maximum principle, we then derive from
\eqref{eq-exterior-inequality}
and \eqref{eq-exterior-supersolution}
 that
\begin{equation}
 \label{eq-tail-function-decay}
 0<y_n(r)\le\Psi(r)\le Ce^{-cr},
 \ \ r\ge R_0,
\end{equation}
where $C>0$ is independent of $n>0$.
We also deduce from \eqref{eq-tail-monotonicity}
and \eqref{eq-exterior-inequality} that
\begin{equation}
 \label{eq-tail-convexity}
 y_n''(r)
 \ge-\frac{N-1}{r}y_n'(r)+\frac12y_n(r)>0,
 \ \ r>R_0.
\end{equation}
It then follows from \eqref{eq-tail-function-decay} and \eqref{eq-tail-convexity}
 that
\[
 \begin{aligned}
  \abs{u_n'(r)}
  &=-y_n'(r)
  \le-\int_{r-1}^r y_n'(s)\,ds\\
  &=y_n(r-1)-y_n(r)
  \le Ce^{-cr},\ \ r\ge R_0+1.
 \end{aligned}
\]
Applying \eqref{eq-Qk-local-convergence} on $[R_0,R_0+1]$,
we may increase $C>0$ so that
\begin{equation}
 \label{eq-uniform-exponential-tail}
 \abs{u_n(r)}+\abs{u_n'(r)}
 \le Ce^{-cr}, \ \ r\ge R_0,
\end{equation}
where $C>0$ and $c>0$ are independent of $n>0$ and $r>0$.
Setting $n\to\infty$ for
\eqref{eq-uniform-exponential-tail},
we then deduce from \eqref{eq-Qk-local-convergence} that
\begin{equation}
 \label{eq-Qk-exponential-tail}
 \abs{Q_k(r)}+\abs{Q_k'(r)}
 \le Ce^{-cr}, \ \ r\ge R_0.
\end{equation}

We finally prove \eqref{eq-sequential-convergence}.
For any fixed $R>R_0$, splitting $\R^N$
into $B_R:=\{x\in\R^N:\abs{x}<R\}$ and $\R^N\setminus B_R$,
we deduce from \eqref{eq-Qk-local-convergence},
\eqref{eq-uniform-exponential-tail}
and \eqref{eq-Qk-exponential-tail} that
\[
 \begin{aligned}
  \limsup_{n\to\infty}
  \norm{u_n-Q_k}_{H^1(\R^N)}^2
  \le C\int_R^{+\infty}e^{-2cr}r^{N-1}\,dr,
   \end{aligned}
\]
\[
 \begin{aligned}
  \limsup_{n\to\infty}
  \norm{u_n-Q_k}_{L^\infty(\R^N)}
  \le2Ce^{-cR},
 \end{aligned}
\]
where $C>0$ is independent of $R$.
Setting $R\to+\infty$, we then obtain that
\eqref{eq-sequential-convergence} holds.
This therefore completes the proof of Proposition~\ref{prop-global-compactness}.
\end{proof}

We finally prove Theorem~\ref{thm-main}
by applying Proposition~\ref{prop-global-compactness}.

\begin{proof}[Proof of Theorem~\ref{thm-main}]
For each integer $k\ge1$, it follows from
Proposition~\ref{prop-global-compactness} that
\eqref{eq-uniform-convergence} holds.
We then deduce from \eqref{eq-uniform-convergence} that
there exists $\varepsilon_k\in(0,\varepsilon_*)$ such that for any $0<\varepsilon<\varepsilon_k$, 
\[
 \sup_{u_\varepsilon\in\mathcal B_{\varepsilon,k}}
 \left[
  \norm{u_\varepsilon-Q_k}_{H^1(\R^N)}
  +\norm{u_\varepsilon-Q_k}_{L^\infty(\R^N)}
 \right]\le1. 
\]
Applying the triangle inequality, we then obtain that for any $0<\varepsilon<\varepsilon_k$
and any $u_\varepsilon\in\mathcal B_{\varepsilon,k}$, 
\[
 \begin{aligned}
  \norm{u_\varepsilon}_{H^1(\R^N)}
  +\norm{u_\varepsilon}_{L^\infty(\R^N)}
  &\le\norm{Q_k}_{H^1(\R^N)}
  +\norm{Q_k}_{L^\infty(\R^N)}+1
  =C_k<+\infty.
 \end{aligned}
\]
This therefore proves \eqref{eq-uniform-bounds}, 
which completes the proof of Theorem~\ref{thm-main}.
\end{proof}

\section{Proof of Theorem \ref{maintheorem}}
\label{sec-uniqueness}

The purpose of this section is to complete the proof of Theorem~\ref{maintheorem}.
Recall from \eqref{eq-Lk} that the radial
linearized operator $L_k:\Hrad\to\Hmrad$
around the $k$-node bound state $Q_k$
of \eqref{eq-limit} is given by
\[
 L_k=-\Delta+1-p\abs{Q_k(x)}^{p-1}
 \,\ \ \text{in}\ \, \R^N, 
\]
where $\Hrad$ is defined in \eqref{1:radial},
and $\Hmrad:=(\Hrad)^*$.
For each integer $k\ge1$, we shall derive
the inverse estimate \eqref{1:eq-Lk}
for $L_k$
around $Q_k$.
We then apply this inverse estimate to prove the uniqueness
of solutions
$u_\varepsilon\in\Hrad\cap L^\infty(\R^N)$
for \eqref{mainequation}
in a fixed $L^\infty(\R^N)$ neighborhood of $Q_k$, where
$\varepsilon>0$ is sufficiently small.
Combining this local uniqueness with
the uniform $L^\infty(\R^N)$ convergence
of Theorem~\ref{thm-main},
we finally conclude that \eqref{mainequation}
has a unique $k$-node bound state
$u_\varepsilon\in\mathcal B_{\varepsilon,k}$
for any sufficiently small $\varepsilon>0$.

To derive the inverse estimate \eqref{1:eq-Lk} of $L_k$ in $\R^N$,
we first prove that $Q_k$ is nondegenerate in $\Hrad$.

\begin{lem}
 \label{prop-Qk-nondegenerate}
 Suppose $N\ge3$ and $1<p<\frac{N+2}{N-2}$.
 Then for each integer $k\ge1$, the $k$-node bound state $Q_k$
 of \eqref{eq-limit} is nondegenerate in $\Hrad$, 
 in the sense that
 \begin{equation}
  \label{eq-Qk-kernel}
  \ker_{\Hrad}L_k=\{0\}.
 \end{equation}
\end{lem}

\begin{proof}
For any $\gamma>0$, denote $U=U(r,\gamma)$
the solution of the following initial value problem:
\begin{equation}
 \label{eq-pure-shooting}
 \begin{cases}
  U_{rr}+\frac{N-1}{r}U_r-U+\abs{U}^{p-1}U=0
  \,\ \ \text{in}\ \ r>0,\\
  U(0,\gamma)=\gamma,
  \ \ U_r(0,\gamma)=0.
 \end{cases}
\end{equation}
Define $\gamma_k:=Q_k(0)>0$ and
\begin{equation}
 \label{eq-Zk}
 Z_k(r):=
 \left.\frac{\partial U(r,\gamma)}{\partial\gamma}
 \right|_{\gamma=\gamma_k},
 \ \ r\ge0.
\end{equation}
Differentiating \eqref{eq-pure-shooting}
with respect to $\gamma$ at $\gamma=\gamma_k$,
we obtain that
\begin{equation}
 \label{eq-Zk-equation}
 \begin{cases}
  Z_k''+\frac{N-1}{r}Z_k'-Z_k
  +p\abs{Q_k}^{p-1}Z_k=0
  \,\ \ \text{in}\ \ r>0,\\
  Z_k(0)=1,
  \ \ Z_k'(0)=0.
 \end{cases}
\end{equation}
It then follows from
\cite[Lemma~3.1(b) and Section~5.8]{tang2026uniqueness} that
\begin{equation}
 \label{eq-Zk-divergence}
 \abs{Z_k(r)}\to+\infty
 \ \ \text{as}\ \ r\to+\infty.
\end{equation}

Suppose $\phi\in\Hrad$ satisfies $L_k\phi=0$ in $\R^N$.
Applying the elliptic regularity to the equation $L_k\phi=0$ in $\R^N$,
we obtain that $\phi\in C^2_{\mathrm{loc}}(\R^N)$.
Since $\phi(x)$ is radially symmetric about the origin,
$\phi=\phi(r)$ satisfies the following ordinary differential equation
\[
 \begin{cases}
  \phi''+\frac{N-1}{r}\phi'-\phi
  +p\abs{Q_k}^{p-1}\phi=0
  \,\ \ \text{in}\ \ r>0,\\
  \phi'(0)=0.
 \end{cases}
\]
Applying the ODE uniqueness theorem,
we then conclude from \eqref{eq-Zk-equation} that
\[
 \phi(r)=\phi(0)Z_k(r),\ \ r\ge0.
\]
Since $\phi\in L^2(\R^N)$, we deduce from
\eqref{eq-Zk-divergence} that $\phi(0)=0$,
and hence $\phi(r) \equiv 0$ for any $r\ge0$.
This proves \eqref{eq-Qk-kernel},
which completes the proof of
Lemma~\ref{prop-Qk-nondegenerate}.
\end{proof}

We now employ Lemma~\ref{prop-Qk-nondegenerate}
to prove the inverse estimate of $L_k$ in $\R^N$. 

\begin{lem}
 \label{lem-Lk-inverse}
 For each integer $k\ge1$, there exists $A_k>0$ such that
 \begin{equation}
  \label{eq-Lk-inverse}
  \norm{\phi}_{H^1(\R^N)}
  \le A_k\norm{L_k\phi}_{H^{-1}(\R^N)},
  \ \ \forall \, \phi\in\Hrad.
 \end{equation}
\end{lem}

\begin{proof}
On the contrary, suppose \eqref{eq-Lk-inverse}
is false for some integer $k\ge1$.
Then there exists a sequence $\{\phi_n\}\subset\Hrad$
such that
\[
 \norm{\phi_n}_{H^1(\R^N)}\equiv 1
 \ \ \text{for any}\ \ n\ge1,
 \ \ \norm{L_k\phi_n}_{H^{-1}(\R^N)}\to0
 \ \ \text{as}\ \ n\to\infty.
\]
Passing to a subsequence if necessary, we may assume that
\[
 \phi_n\rightharpoonup\phi
 \,\ \ \text{in}\ \ H^1(\R^N),
 \ \ \phi_n\to\phi
 \,\ \ \text{in}\ \ L^2_{\mathrm{loc}}(\R^N)
 \ \ \text{as}\ \ n\to\infty,
\]
where $\phi\in\Hrad$.
It then follows from \eqref{eq-Lk} that $L_k\phi=0$ in $\R^N$.
Lemma~\ref{prop-Qk-nondegenerate} further yields that
$\phi \equiv 0$ in $\R^N$. 

Since $\phi_n\to0$ in $L^2_{\mathrm{loc}}(\R^N)$
as $n\to\infty$ and
$\norm{\phi_n}_{H^1(\R^N)}\equiv 1$ holds for any $n\ge1$,
the decaying property of $Q_k(x)$
as $\abs{x}\to+\infty$ yields that
\[
 \int_{\R^N}\abs{Q_k(x)}^{p-1}\phi_n(x)^2\,dx
 \to0
 \ \ \text{as}\ \ n\to\infty.
\]
It then follows from \eqref{eq-Lk} that
\[
 \begin{aligned}
  1
  &=\norm{\phi_n}_{H^1(\R^N)}^2\\
  &=\langle L_k\phi_n,\phi_n\rangle
  +p\int_{\R^N}
  \abs{Q_k(x)}^{p-1}\phi_n(x)^2\,dx
  \to0
  \ \ \text{as}\ \ n\to\infty,
 \end{aligned}
\]
a contradiction.
This therefore proves \eqref{eq-Lk-inverse}
and completes the proof of Lemma~\ref{lem-Lk-inverse}.
\end{proof}

Applying Lemma~\ref{lem-Lk-inverse},
we prove the following result on the uniqueness
of solutions for \eqref{mainequation}
in an $L^\infty(\R^N)$ neighborhood of $Q_k$.

\begin{prop}
 \label{prop-local-uniqueness}
 Suppose $p,q>1$ satisfy \eqref{eq-range}.
 Then there exist
 $\eta_k>0$ and $\varepsilon_k\in(0,\varepsilon_*)$
 such that for any $0<\varepsilon<\varepsilon_k$,
 \eqref{mainequation} has at most one solution
 $u_\varepsilon\in\Hrad\cap L^\infty(\R^N)$
 satisfying
 \begin{equation}
  \label{eq-local-neighborhood}
  \norm{u_\varepsilon-Q_k}_{L^\infty(\R^N)}<\eta_k.
 \end{equation}
\end{prop}

\begin{proof}
Define for each integer $k\ge1$, 
\[
 B_k:=\norm{Q_k}_{L^\infty(\R^N)}+1>0, 
\]
and
\[
 \omega_k(\rho)
 :=\sup_{\substack{a,b\in[-B_k,B_k]\\\abs{a-b}\le\rho}}
 p\left|\abs{a}^{p-1}-\abs{b}^{p-1}\right|,
 \ \ \rho\ge0.
\]
Since $t\mapsto p\abs{t}^{p-1}$ is uniformly continuous
on $[-B_k,B_k]$, we get that
\[
 \omega_k(\rho)\to0
 \ \ \text{as}\ \ \rho\searrow0.
\]
We therefore choose $\eta_k\in(0,1)$ and
$\varepsilon_k\in(0,\varepsilon_*)$ such that
\begin{equation}
 \label{eq-local-smallness}
 A_k\left[
 \omega_k(\eta_k)+qB_k^{q-1}\varepsilon_k
 \right]\le\frac12,
\end{equation}
where $A_k>0$ is given by Lemma~\ref{lem-Lk-inverse}.

Suppose $0<\varepsilon<\varepsilon_k$, and
$u_\varepsilon,w_\varepsilon\in\Hrad\cap L^\infty(\R^N)$ are solutions of
\eqref{mainequation} satisfying
\eqref{eq-local-neighborhood}.
Define 
\[
z_\varepsilon(x):=u_\varepsilon(x)-w_\varepsilon(x) \,\ \ \text{in}\ \ \R^N.
\]
We also define for any $s>1$, 
\[
 \mathfrak a_{\varepsilon,s}(x)
 :=s\int_0^1\abs{(1-t)w_\varepsilon(x)+tu_\varepsilon(x)}^{s-1}\,dt
 \,\ \ \text{in}\ \ \R^N.
\]
Applying the integral mean value formula,
we then derive from \eqref{mainequation} and \eqref{eq-Lk} that
\begin{equation}
 \label{eq-difference-equation}
 L_kz_\varepsilon=M_{\varepsilon}z_\varepsilon  \,\ \ \text{in}\ \ \R^N, 
\end{equation}
where
\begin{equation}
 \label{eq-difference-multiplier}
 M_{\varepsilon}(x)
 :=\mathfrak a_{\varepsilon,p}(x)-p\abs{Q_k(x)}^{p-1}
 -\varepsilon\mathfrak a_{\varepsilon,q}(x)
 \,\ \ \text{in}\ \ \R^N.
\end{equation}
Note from \eqref{eq-local-neighborhood} that for any $x\in\R^N$ and $t\in[0,1]$, 
\[
 \begin{aligned}
  \abs{(1-t)w_\varepsilon(x)+tu_\varepsilon(x)}&\le B_k,
 \end{aligned}
\]
\[
  \begin{aligned}
  \abs{(1-t)w_\varepsilon(x)+tu_\varepsilon(x)-Q_k(x)}<\eta_k.
 \end{aligned}
\]
The definitions of $\mathfrak a_{\varepsilon,s}(x)$ 
and $\omega_k(\rho)$ then yield that
\[
 \begin{aligned}
  \norm{\mathfrak a_{\varepsilon,p}-p\abs{Q_k}^{p-1}}_{L^\infty(\R^N)}
  \le\omega_k(\eta_k),\ \
  \norm{\mathfrak a_{\varepsilon,q}}_{L^\infty(\R^N)}
  \le qB_k^{q-1}.
 \end{aligned}
\]
We thus deduce from \eqref{eq-difference-multiplier} that
\begin{equation}
 \label{eq-multiplier-small}
 \norm{M_{\varepsilon}}_{L^\infty(\R^N)}
 \le\omega_k(\eta_k)+qB_k^{q-1}\varepsilon.
\end{equation}

We now get that for any $M\in L^\infty(\R^N)$ and
$\phi\in H^1(\R^N)$, 
\begin{equation}
 \label{eq-multiplier-Hminus1}
 \norm{M\phi}_{H^{-1}(\R^N)}
 \le\norm{M\phi}_{L^2(\R^N)}
 \le\norm{M}_{L^\infty(\R^N)}\norm{\phi}_{H^1(\R^N)}.
\end{equation}
Since $z_\varepsilon\in\Hrad$, we derive from
\eqref{eq-Lk-inverse}, \eqref{eq-difference-equation}
and \eqref{eq-multiplier-Hminus1} that
\[
 \begin{aligned}
  \norm{z_\varepsilon}_{H^1(\R^N)}
  &\le A_k\norm{M_{\varepsilon}z_\varepsilon}_{H^{-1}(\R^N)}\\
  &\le A_k\norm{M_{\varepsilon}}_{L^\infty(\R^N)}
  \norm{z_\varepsilon}_{H^1(\R^N)}\\
  &\le\frac12\norm{z_\varepsilon}_{H^1(\R^N)},
 \end{aligned}
\]
where the last inequality follows from
\eqref{eq-local-smallness} and \eqref{eq-multiplier-small}.
We thus obtain that $z_\varepsilon(x) \equiv 0$ in $\R^N$, and hence $u_\varepsilon \equiv w_\varepsilon$ in $\R^N$. 
This completes the proof of
Proposition~\ref{prop-local-uniqueness}.
\end{proof}

We are next ready to prove Theorem~\ref{maintheorem} as follows. 

\begin{proof}[Proof of Theorem~\ref{maintheorem}]
For each integer $k\ge1$, it follows from
\cite[Theorem~1.1]{CortazarGarciaHuidobroYarur2013}
that $\mathcal B_{\varepsilon,k}\ne\varnothing$ holds
for any $0<\varepsilon<\varepsilon_*$.
Applying Proposition~\ref{prop-local-uniqueness},
we deduce from \eqref{eq-uniform-convergence} that 
there exists an $\varepsilon_k\in(0,\varepsilon_*)$ 
such that for any $0<\varepsilon<\varepsilon_k$, $\mathcal B_{\varepsilon,k}$
contains exactly one element.
This proves the uniqueness, up to translation
and reflection, of $k$-node bound states for \eqref{mainequation}, 
which therefore completes the proof of
Theorem~\ref{maintheorem}.
\end{proof}

Finally, the following proposition shows that 
for each integer $k\ge1$, the unique bound state
$u_\varepsilon\in\mathcal B_{\varepsilon,k}$
of \eqref{mainequation}
is nondegenerate in $\Hrad$
for any sufficiently small $\varepsilon>0$.
Define the radial linearized operator
$\mathcal L_{\varepsilon,k}:\Hrad\to\Hmrad$
around $u_\varepsilon$ by
\begin{equation}
 \label{eq-perturbed-linearized}
 \mathcal L_{\varepsilon,k}
 :=-\Delta+1-p\abs{u_\varepsilon(x)}^{p-1}
 +\varepsilon q\abs{u_\varepsilon(x)}^{q-1}
  \,\ \ \text{in}\ \, \R^N.
\end{equation}

\begin{prop}
 \label{cor-perturbed-nondegeneracy}
 Suppose $p,q>1$ satisfy \eqref{eq-range}.
 Then for any sufficiently small $\varepsilon>0$, 
 the bound state $u_\varepsilon\in\mathcal B_{\varepsilon,k}$
 of \eqref{mainequation} is nondegenerate in $\Hrad$,
 in the sense that
 \begin{equation}
  \label{eq-u-kernel}
  \ker_{\Hrad}\mathcal L_{\varepsilon,k}=\{0\}.
 \end{equation}
\end{prop}

\begin{proof}
For each integer $k\ge1$, suppose $\phi\in\Hrad$
satisfies $\mathcal L_{\varepsilon,k}\phi=0$ in $\R^N$.
It then follows from \eqref{eq-Lk}
and \eqref{eq-perturbed-linearized} that
\begin{equation}
 \label{eq-perturbed-kernel-equation}
 L_k\phi=B_{\varepsilon,k}\phi \,\ \ \text{in}\ \ \R^N, 
\end{equation}
where
\[
 B_{\varepsilon,k}(x)
 :=p\left[
 \abs{u_\varepsilon(x)}^{p-1}
 -\abs{Q_k(x)}^{p-1}
 \right]
 -\varepsilon q\abs{u_\varepsilon(x)}^{q-1}
 \,\ \ \text{in}\ \ \R^N.
\]
Note from \eqref{eq-uniform-bounds}
and \eqref{eq-uniform-convergence} that for any sufficiently small $\varepsilon>0$,
\[
 \norm{u_\varepsilon}_{L^\infty(\R^N)}\le C_k,
 \ \ \norm{Q_k}_{L^\infty(\R^N)}\le C_k, 
\]
where $C_k>0$ is independent of $\varepsilon>0$.
Using the uniform continuity of
$t\mapsto\abs{t}^{p-1}$ on $[-C_k,C_k]$,
we then deduce from \eqref{eq-uniform-convergence} that
\begin{equation}
 \label{eq-perturbation-small}
 \begin{aligned}
  \norm{B_{\varepsilon,k}}_{L^\infty(\R^N)}
  &\le p\norm{
  \abs{u_\varepsilon}^{p-1}
  -\abs{Q_k}^{p-1}
  }_{L^\infty(\R^N)}
  +\varepsilon qC_k^{q-1}
  \to0
  \ \ \text{as}\ \ \varepsilon\searrow0.
 \end{aligned}
\end{equation}

Applying \eqref{eq-Lk-inverse} to $\phi\in\Hrad$,
we then obtain from \eqref{eq-multiplier-Hminus1}, \eqref{eq-perturbed-kernel-equation}
and \eqref{eq-perturbation-small} that
for any sufficiently small $\varepsilon>0$,
\[
 \begin{aligned}
  \norm{\phi}_{H^1(\R^N)}
  &\le A_k
  \norm{B_{\varepsilon,k}\phi}_{H^{-1}(\R^N)}\\
  &\le A_k
  \norm{B_{\varepsilon,k}}_{L^\infty(\R^N)}
  \norm{\phi}_{H^1(\R^N)}
  \le\frac12\norm{\phi}_{H^1(\R^N)}. 
 \end{aligned}
\]
We thus have $\phi \equiv 0$ in $\R^N$, which proves
\eqref{eq-u-kernel}. This therefore completes the proof of
Proposition~\ref{cor-perturbed-nondegeneracy}.
\end{proof}

\end{document}